\documentclass[11pt]{amsart}
\usepackage{amsfonts}
\usepackage{cite}
\usepackage{amssymb}
\usepackage{mathrsfs}
\usepackage{amsthm,amsmath,amssymb}
\numberwithin{equation}{section}
\usepackage{dsfont}
\usepackage[colorlinks=true]{hyperref}
\usepackage{color}

\DeclareMathOperator{\supp}{supp}

\hypersetup{urlcolor=red, citecolor=blue}
\theoremstyle{plain}
\newtheorem{theorem}{Theorem}[section]
\newtheorem{corollary}[theorem]{Corollary}

\newtheorem{lemma}[theorem]{Lemma}
\newtheorem{proposition}[theorem]{Proposition}

\theoremstyle{definition}
\newtheorem{definition}[theorem]{Definition}
\newtheorem{remark}[theorem]{Remark}

\newcommand{\GC}{\Gamma}

\newcommand{\R}{\mathbb{R}}

\begin{document}
 	\title[ Qualitative properties of positive solutions to  ]{Qualitative properties of positive solutions to mixed local and nonlocal  critical problems in $\mathbb{R}^n$}

\author{Xifeng Su}
\address{School of Mathematical Sciences, Laboratory of Mathematics and Complex Systems (Ministry of Education)\\
	Beijing Normal University,
	No. 19, XinJieKouWai St., HaiDian District, Beijing 100875, P. R. China}
\email{xfsu@bnu.edu.cn\\ billy3492@gmail.com}

\author{Shasha Xu}
\address{School of Mathematical Sciences\\
	Beijing Normal University,
	No. 19, XinJieKouWai St., HaiDian District, Beijing 100875, P. R. China}
\email{ssxu@mail.bnu.edu.cn\\nanqiqi2020@163.com}


\thanks{X. Su is supported by the National Natural Science Foundation of China (Grant No. 12371186, 12571207).}	
	
	\begin{abstract}
	We consider  the following mixed local and non-local  critical elliptic equation:
		\begin{equation*}\label{0.1}
			\left\{
			\begin{array}{lll}
				-\Delta u+(-\Delta)^su=\lambda h u^{p}+u^{2^*-1}, &\text{in}\,\, \mathbb{R}^n, \\
                                  u>0,      &\text {in} \,\, \mathbb{R}^n,\\
                                 \lim\limits_{|x|\to\infty} u(x) = 0,
			\end{array}
			\right.
		\end{equation*}
where $n\geqslant4, \,\, p\in (0,2^*-1),\,\, 2^*:=\frac{2n}{n-2}$ and $h$ is a positive function.

We first show the existence and regularity results  of  viscosity solutions to the above critical elliptic equation. More precisely, from \cite{Su-Xu} weak solutions are obtained and we prove they are indeed viscosity solutions and their regularity is: \( u \in C^{\alpha}(\mathbb{R}^n) \) for $p\in(0,1);$  \( u \in C^{2,\beta}(\mathbb{R}^n) \) for $p\in [1, 2^*-1).$ 

Moreover, for $p\in [1, 2^*-1)$, these viscosity solutions are indeed classical ones and we then prove the existence of positive solutions with the qualitative properties such as
the decay estimates and the radial symmetry.
\end{abstract}
	
	\maketitle
	\tableofcontents
	\section{Introduction}	
In this paper, we are concerned with qualitative properties of positive solutions to the mixed local-nonlocal critical elliptic equation:	
\begin{equation}\label{equ1}
			\mathcal{L}u:=-\Delta u+(-\Delta)^su=\lambda h u^{p}+u^{2^*-1},\quad  \text{ in }\,\, \mathbb{R}^n
	\end{equation}
satisfying $u>0$ in $\mathbb{R}^n$ and $\lim\limits_{|x|\to+\infty}u(x)=0,$ where $n\geqslant4, \,\, p\in (0,2^*-1)$\footnote{The various value of $p$ means that  the subcritical perturbations could be sublinear, linear and superlinear.}, $\lambda$ is  some appropriate positive real number\footnote{ See \cite{Su-Xu} for more details about the existence results of weak solutions.} and the weight function  $h$ satisfies:	 
\begin{itemize}
\item [(h$_1$)]	$ 0<h\in L^{\infty}(\mathbb{R}^n)\cap L^1(\mathbb{R}^n)\cap C^1(\R^n)$ such that  the embedding $\mathcal{X}^{1,2}(\mathbb{R}^n)\hookrightarrow\hookrightarrow L^{p+1}(\mathbb{R}^n,hdx)$ is compact, where $\mathcal{X}^{1,2}(\mathbb{R}^n)$ is an appropriate function space defined in \eqref{function space}.
\item [(h$_2$)] The function \(h\) is radially symmetric about \(x_0\) and nonincreasing in the radial variable, namely,
$
h(x)=h(r),\quad r=|x-x_0|,
\quad h'(r)\le0\quad(r>0).
$
\end{itemize}

Here, for any $s\in(0,1)$ the fractional Laplacian is defined as $$(-\Delta)^su(x)=c_{n,s}\text{P.V.}\int_{\mathbb{R}^n}\frac{u(x)-u(y)}{|x-y|^{n+2s}}dy,$$
where $c_{n,s}>0$ is a suitable normalization constant, whose explicit value does not play a role here, and P.V. means that the integral is taken in the Cauchy principal value sense.	

The pseudo-differential operator $-\Delta+(-\Delta)^s$ is the infinitesimal generator of a stochastic process $X$, where $X$ is the mixture of a Brownian motion and an independent symmetric $2s$-stable Lévy process.
These operators can be characterized using the Fourier transform $\mathcal{F}$ as
\[
\mathcal{F} \left( -\Delta u + (-\Delta)^s u \right) (\xi) = \left( |\xi|^2 + |\xi|^{2s} \right) \mathcal{F} u (\xi).
\]
We can thus define a heat kernel $\mathcal{H}(x, t)$ associated with the above operator as
\begin{equation}\label{heat kernel}
\mathcal{H}(x, t) := \int_{\mathbb{R}^n} e^{-t(|\xi|^2 + |\xi|^{2s}) + 2\pi i x \cdot \xi}  d\xi
\end{equation}
for $t > 0$ and $x \in \mathbb{R}^n$. 

Note that since equation~\eqref{equ1} involves non-scale-invariance of operators of different orders, the scale-invariant techniques fail in the present setting.
We will overcome this difficulty by 
 decomposing the Riesz kernel  from the above heat kernel of the mixed-order operator, which would provide a different approach from the one in \cite{SerenaDipierro2025} where the mixed subcritical nonlinear Schr\"odinger equations are considered.


From \cite{Su-Xu}, we have the existence results of weak solutions of \eqref{equ1} as follows:
	\begin{description}
	\item [Sublinear case $p\in(0,1)$]  there exists at least two nonnegative solution for $\lambda>0$ sufficiently small;
	\item [Linear case $p=1$]  there exists at least a nonnegative solution for $\lambda>0$ sufficiently small  where $\lambda_1$ is the first eigenvalue of the mixed operator;
	\item [Superlinear case $p\in(1, 2^*-1)$]  there exists at least a nonnegative solution for $\lambda<\lambda_1.$
	\end{description}

Based on these results, we always assume that  we have the existence of weak solutions in the present paper.
Our goal is to bootstrap the regularity of weak solutions according to different types of subcritical perturbations and investigate qualitative properties of positive classical solutions to \eqref{equ1} in the case of both linear and superlinear perturbations. Our strategy is: 
\begin{itemize}
    \item We first obtain the H\"older continuity of weak solutions for $p\in(0,1)$ and $C^{2,\alpha}$-regularity for $p\in [1, 2^*-1).$ The methods of proof rely on the scaling and iterative techniques, the heat kernel properties of the mixed operator  and the Riesz kernel properties of the fractional Laplacian.
    Consequently, we deduce that these weak solutions are viscosity solutions.
    \item We then show the existence of classical positive solutions for $p\in [1, 2^*-1)$ and the qualitative properties of such positive solutions, such as the power-type decay at infinity and the radial symmetry.
\end{itemize}



We point out that weak solutions of mixed critical problem presented here are proved to be viscosity solutions. 
Regarding the general theory of viscosity solutions for both the integer-order and fractional Laplace equations, one may refer to e.g. \cite{MR3161511,MR2494809,MR2707618}.

We now state the main results of this paper.
\begin{theorem}\label{thm:vis}
Let $u\in\mathcal{X}^{1,2}(\mathbb{R}^n)$ be a weak solution of \eqref{equ1} with $p\in(0,2^{*}-1)$. Then, $u\in C^{\alpha}(\R^n)$ for any $\alpha \in (0, \min\{2s, 1\}).$ Moreover,  if $u$ is a weak solution of \eqref{equ1}. Then, $u$ is  a viscosity solution of \eqref{equ1}.
\end{theorem}

\begin{remark}
The regularity assumption of $h\in C^1(\mathbb{R}^n)$ in (h$_1$) is only used to prove the $C^{2,\alpha}$-regularity of weak solutions of \eqref{equ1} with $p\in[1,2^*-1)$. However,  for $p\in(0,1)$, 
such assumption could be dropped to guarantee the H\"older continuity of the  corresponding weak solutions.
\end{remark}

	Furthermore, for the case of $p\in[1,2^*-1)$, the viscosity solutions of \eqref{equ1} become  classical solutions.
	
	\begin{theorem}[$C^{2,\alpha}$- regularity for $p\in[1,2^*-1)$]\label{thm:1.2}
	Let $u \in \mathcal{X}^{1,2}(\mathbb{R}^n)$ be a weak solution of \eqref{equ1} with $p\in[1,2^*-1)$. Then $u \in C^{2,\alpha}(\mathbb{R}^n)$ for any $\alpha \in (0, 1)$ and $$\|u\|_{C^{2,\alpha}(\mathbb R^n)}
\leqslant C\Bigl(\lambda\|h\|_{C^1(\mathbb R^n)}
\|u\|_{C^1(\mathbb R^n)}^p+\|u\|_{C^1(\mathbb R^n)}^{2^*-1}
+\|u\|_{L^\infty(\mathbb R^n)}
\Bigr),$$ for some $C=C(n,s,p,\alpha, 2^*).$ 
	Moreover, $\lim_{|x|\to+\infty}u(x)=0.$
	\end{theorem}
	\begin{remark}Unlike the subcritical problems consider e.g. in \cite{MR4808805}, here the control constant $C$ in Theorem~\ref{thm:1.2} is also influenced by the critical exponent.
	\end{remark}
	Regarding the regularity of solutions to either the classical Laplace or fractional Laplace equations on $\mathbb{R}^n$, there are already quite a number of results, see e.g. \cite{MR2270163,CABRE201423,MR2863859}. Especially, in \cite{dipierro2024global}, the H\"older regularity for the classical fractional Laplacian equation with perturbations has been established with $p\in(0,1)$. Here, we prove the corresponding  H\"older or $C^{2,\alpha}-$regularity for mixed local and nonlocal equations driven by the sum of Laplacian and fractional Laplacian with $p\in(0,2^*-1]$.
	
		Finally, we obtain the power-type decay at infinity and the radial symmetry of positive classical solutions to \eqref{equ1} by comparison arguments and the method of moving planes, respectively.
	\begin{theorem} [Qualitative properties of classical positive solutions]\label{thm4.8}
	 Let \( n \geqslant4,\, p\in[1,2^*-1) \) and \( s \in (0,1) \). Let \( u \) be a positive classical solution of \eqref{equ1}. Then, there exist constants \( C_1, C_2>0 \) such that, for every \( |x| \geqslant 1 \),
\[
\frac{C_1}{|x|^{n-2s}} \leqslant u(x) \leqslant \frac{C_2}{|x|^{n-2s}}.
\]
Moreover, if \((h_2)\) holds, then every positive bounded classical solution \(u\in \mathcal{X}^{1,2}(\mathbb R^n)\) of \eqref{equ1} is radially symmetric about \(x_0\) and nonincreasing with respect to \(r\).
\end{theorem}

Note that compared with the power-type decay in \cite{SerenaDipierro2025}, we could only obtain the order of the upper bound $\frac{1}{ |x|^{n-2s}}$ here  rather than $\frac{1}{ |x|^{n+2s}}$ for the Schr\"odinger equation in \cite{SerenaDipierro2025}.

The rest of this paper is organized as follows. In Section \ref{sec2}, we collect some elementary results of $\mathcal{X}^{1,2}(\mathbb{R}^n)$, introduce the functional setting, such as the notions of weak solutions, viscosity solutions  and energy functional, and several fundamental properties of the heat kernel. 

Sections \ref{sec3}-\ref{sec:existence of vs} are devoted to the proofs of existence of viscosity solutions and  their regularity (Theorems \ref{thm:vis} and \ref{thm:1.2}).

In Sections~\ref{sec5}-\ref{sec6}, we establish the power-type decay estimate at infinity and the radial  symmetry of positive classical solutions (Theorem \ref{thm4.8}).

	\section{Preliminaries}\label{sec2}
	In this section, we will introduce the functional setting for the mixed local and nonlocal critical elliptic problem \eqref{equ1} and the properties of heat kernel associated to the mixed operator given in \eqref{heat kernel}. 
	
	\subsection{The functional setting}\label{sec2.1}
	Let $s\in(0,1)$. For a measurable function \(u:\mathbb R^n\to\mathbb R\), we define its Gagliardo seminorm by
\[
[u]_s:=
\left(
\iint_{\mathbb R^{2n}}
\frac{|u(x)-u(y)|^2}{|x-y|^{n+2s}}\,dx\,dy
\right)^{1/2}.
\]
	We define the function space 
	\begin{equation}\label{function space}
	\mathcal{X}^{1,2}(\mathbb{R}^n) : = \left\{u\in L^{2^*}(\mathbb{R}^n):\triangledown u\in L^2(\mathbb{R}^n)\, \text{and}\, [u]_s< +\infty \right\}
	\end{equation}
	which is the completion of $C_0^{\infty}(\mathbb{R}^n)$ with respect to the norm 
	$$\Vert u\Vert:=\left(\Vert\triangledown u\Vert_{L^2(\mathbb{R}^n)}^2+[u]_s^2\right)^{1/2}, \quad  u\in C_0^{\infty}(\mathbb{R}^n).$$

 In the literature, different notions of solutions are taken into account when dealing with elliptic equations, such as the weak (also called distributional, or variational, or energy) solutions (i.e. the solutions that belong to a suitable Sobolev space and satisfy the equation in a distributional sense, when integrated against a suitable set of test functions) and the viscosity solutions (i.e. all the smooth functions that touch either from above or below the continuous solution are required to be either  viscosity subsolutions or viscosity supersolutions). 	
\begin{definition}[Weak solution]We say that \(u \in \mathcal{X}^{1,2}(\mathbb{R}^n)\) is a (weak) solution of \eqref{equ1} if
\begin{equation}\label{equu:2.5}
		\begin{aligned}
			&\int_{\mathbb{R}^n}\triangledown u(x)\cdot\triangledown \varphi(x)\, dx+\frac{c_{n,s}}{2}\iint_{\mathbb{R}^{2n}}\frac{\left(u(x)-u(y)\right)\left(\varphi(x)-\varphi(y)\right)}{|x-y|^{n+2s}}\, dxdy\\
			&\,\,\,\,\,=\int_{\mathbb{R}^n}|u|^{2^*-2}u \varphi \, dx+\lambda\int_{\mathbb{R}^n}h|u|^{p-1}u \varphi \, dx
		\end{aligned}
	\end{equation}
for any $\varphi\in \mathcal{X}^{1,2}(\mathbb{R}^n).$	
\end{definition}

	We recall that \( u \in C(\mathbb{R}^n) \) is a \textit{viscosity subsolution} of $-\Delta u+(-\Delta)^s u=f(x,u(x))\text{ in } \mathbb{R}^n$ if for any $x_0 \in \mathbb{R}^n$ and $\phi \in C^2(B_R(x_0))$ be such that
\[
\phi(x_0) = u(x_0) \quad \text{and} \quad \phi \geqslant u \text{ in } B_R(x_0).
\]
If we define
\[
v(x):= 
\begin{cases}
\phi(x) & \text{if } x \in B_R(x_0), \\
u(x) & \text{otherwise},
\end{cases}
\]
then we have \( -\Delta v(x_0)+(-\Delta)^s v(x_0) \leqslant f(x_0, v(x_0)) \).

Similarly, we say that \( u \in C(\mathbb{R}^n) \) is a \textit{viscosity supersolution} of $-\Delta u+(-\Delta)^s u=f(x,u(x)) \text{ in } \mathbb{R}^n$ if for any \( \psi \in C^2(B_R(x_0)) \) such that \( \psi(x_0) = u(x_0) \) and \( \psi \leqslant u \) in \( B_R(x_0) \), if we define
\[
w(x) := 
\begin{cases} 
\psi(x) &  \text{if } x \in B_R(x_0), \\
u(x) & \text{otherwise} ,
\end{cases}
\]
then we have \( -\Delta w(x_0)+(-\Delta)^s w(x_0) \geqslant f(x_0, w(x_0)) \).

The viscosity formulation can also be extended to upper semicontinuous subsolutions and lower semicontinuous supersolutions; see \cite{MR3161511}.

\begin{definition}
A function \( u \) is called a \textit{viscosity solution} of problem \eqref{equ1} if it is both a viscosity subsolution and a viscosity supersolution.
\end{definition}

\subsection{The properties of heat kernel}
First, we give several  properties related to the heat kernels of the mixed local and nonlocal operators.
\begin{theorem}[{\cite[Theorem 3.1]{SerenaDipierro2025}}]\label{heat}
Let $n \geqslant 1$ and $s \in (0, 1)$. Let $\mathcal{H}$ be as defined in \eqref{heat kernel}. Then,
\begin{itemize}
    \item $\mathcal{H}$ is nonnegative, radially symmetric, and nonincreasing with respect to $r = |x|$.
    \item There exist positive constants $C_1$ and $C_2$ such that
    \begin{equation*}
   \mathcal{H}(x,t) \leqslant C_1 \left\{\frac{t}{|x|^{n+2s}} \vee \frac{t^{s}}{|x|^{n+2s}}\right\} \wedge \left\{t^{-\frac{n}{2s}} \wedge t^{-\frac{n}{2}}\right\}, 
     \end{equation*}
    and
    \[
    \mathcal{H}(x, t) \geqslant C_2 
    \begin{cases} 
    \frac{t}{|x|^{n+2s}} & \text{if } 1 < t < |x|^{2s} \\ 
    e^{\frac{\pi |x|^2}{t}} t^{-\frac{n}{2}} & \text{if } |x|^2 < t < |x|^{2s} < 1, 
    \end{cases}
    \]
    where $a \land b := \min \{a, b\}$ and $a \lor b := \max \{a, b\}$.
\end{itemize}
\end{theorem}
{{
Then, 
define the auxiliary function
\begin{equation}\label{Bessel}\mathcal{Z}(x)=\int_0^{+\infty}\mathcal{H}(x,t)dt. \end{equation}
We give some basic properties of $\mathcal{Z}(x)$, which are essential for proving Theorem \ref{thm4.8}, as summarized below. 
To start with, we show the existence of the upper bound of $\mathcal{Z}(x)$.
\begin{lemma}\label{lem4.1}
Let \( n > 2 \) and \( s \in (0, 1) \).  
Then, 
\[
\begin{split}
\mathcal{Z}(x) &\leqslant 
\begin{cases} 
\frac{c_1}{|x|^{n-2s}} & \text{if } |x| > 1, \\ 
\frac{c_2}{|x|^{n-2}} & \text{if } 0< |x| \leqslant 1, 
\end{cases} 
\end{split}
\]
for some positive constants \( c_1 \) and \( c_2 \) depending on \( n \) and \( s \).
\end{lemma}

\begin{proof}
From Theorem \ref{heat}, we derive that
\[
0 \leqslant \mathcal{H}(x, t) \leqslant {C_1} \left\{ t^{-\frac{n}{2s}} \wedge \frac{t}{|x|^{n+2s}} \right\} \quad\text{for all}\quad  t > 1,
\]
\[ \text{and}\quad
0 \leqslant \mathcal{H}(x, t) \leqslant {C_1} \left\{ t^{-\frac{n}{2}} \wedge \frac{t^s}{|x|^{n+2s}} \right\}\quad\text{for all}\quad  t \in (0, 1).\]
 From these formulas, we conclude that, if \( |x| > 1 \),
\begin{equation}\label{equ:4.1}
\begin{aligned}
\mathcal{Z}(x) &= \int_{1}^{+\infty}   \mathcal{H}(x, t)  dt + \int_{0}^{1}   \mathcal{H}(x, t)  dt \\
&\leqslant {C_1} \left( \int_{1}^{|x|^{2s}}   \frac{t}{|x|^{n+2s}}  dt + \int_{|x|^{2s}}^{+\infty}   t^{-\frac{n}{2s}}  dt + \int_{0}^{1}   \frac{t^s}{|x|^{n+2s}}  dt \right) 
\leqslant \frac{c_1}{ |x|^{n-2s}},
\end{aligned}
\end{equation}
and, if \( |x| \leqslant 1 \),
\begin{equation}\label{equ:4.2}
\begin{aligned}
\mathcal{Z}(x) &= \int_1^{+\infty}   \mathcal{H}(x,t)  dt + \int_0^1   \mathcal{H}(x,t)  dt \\
&\leqslant {C_1} \left( \int_1^{+\infty}   t^{-\frac{n}{2s}}  dt + \int_0^{|x|^2}   \frac{t^s}{|x|^{n+2s}}  dt + \int_{|x|^2}^1   t^{-\frac{n}{2}}  dt \right) \leqslant \frac{c_2}{|x|^{n-2}}.
\end{aligned}
\end{equation}
By combining \eqref{equ:4.1} with \eqref{equ:4.2}, we complete the proof.
\end{proof}

We shall make use of the nonnegativity of \( \mathcal{H} \) and Theorem \ref{heat} to prove the strict positivity of \( \mathcal{Z} \) and obtain its lower bound.

\begin{lemma}\label{lem4.2}
Let \( n \geqslant 2 \) and \( s \in (0, 1) \). Then, \( \mathcal{Z} \) is positive and
\[
\begin{split}
\mathcal{Z}(x) &\geqslant 
\begin{cases} 
\frac{c_3}{|x|^{n-2s}} & \text{if } |x| > 1, \\ 
\frac{c_4}{|x|^{n-2}} & \text{if } 0< |x| \leqslant 1, 
\end{cases} 
\end{split}
\]
for some positive constants \( c_3 \) and \( c_4 \) depending only on \( n \) and \( s \).
\end{lemma}

\begin{proof}
From Theorem \ref{heat}, we know that
\[
\mathcal{H}(x, t) \geqslant C_2 \frac{t}{|x|^{n+2s}} \quad \text{if } 1 < t < |x|^{2s},
\]
and $\mathcal{Z}$ is nonnegative, radially symmetric, and nonincreasing as a function of $r=|x|,$ one deduces that, for any \( |x|^{2s} > 2 \),
\begin{equation}\label{equ:4.3}
\mathcal{Z}(x) = \int_{0}^{+\infty}   \mathcal{H}(x, t)  dt \geqslant  \int_{1}^{|x|^{2s}}   \mathcal{H}(x, t)  dt \geqslant \frac{c_3}{|x|^{n-2s}}.
\end{equation}
For $1<|x|\leqslant2^{\frac{1}{2s}},$ by monotonicity and after decreasing \(c_3\) if necessary, we also obtain $\mathcal{Z}(x)  \geqslant \frac{c_3}{|x|^{n-2s}}.$

Moreover, in light of Theorem \ref{heat}, for every \( |x|^2 < t < |x|^{2s} < 1 \),
\[
\mathcal{H}(x, t) > C_2 e^{\pi \frac{|x|^2}{t}} t^{-\frac{n}{2}}.
\]
As a consequence, by the definition of \( \mathcal{Z} \), one has that, for every \(0< |x| < \frac{1}{2} \),
\begin{equation}\label{equ:4.4}
\begin{aligned}
\mathcal{Z}(x) &\geqslant \int_{|x|^2}^{|x|^{2s}} C_2 e^{\pi \frac{|x|^2}{t}} t^{-\frac{n}{2}}    dt \\
&\geqslant  C_2 \int_{|x|^{2-2s}}^{1} e^{\pi y} |x|^{-n+2} y^{\frac{n}{2}-2}  dy \\
&\geqslant \frac{ C_2}{|x|^{n-2}} \int_{{\frac{1}{2}}^{2-2s}}^{1} e^{\pi y} y^{\frac{n}{2}-2}  dy = \frac{c_4}{|x|^{n-2}}.
\end{aligned}
\end{equation}
Similarly, in light of monotonicity and after decreasing \(c_4\) if necessary, we deduce $\mathcal{Z}(x)\geqslant \frac{c_4}{|x|^{n-2}},$ for $\frac{1}{2}\leqslant |x|\leqslant 1.$
In addition, we recall that \( \mathcal{Z} \geqslant 0 \). According to \eqref{equ:4.3} and \eqref{equ:4.4} , one concludes that \( \mathcal{Z}(x) > 0 \) for every \( x \in \mathbb{R}^n \setminus\{0\}\).
\end{proof}



To sum up, we can obtain the following theorem.
\begin{theorem}\label{kernelz}
Let \( n > 2 \) and \( s \in (0, 1) \). Then, \( \mathcal{Z} \) is positive, radially symmetric, and nonincreasing as a function of $r=|x|,$ and
\[
\frac{c_3}{|x|^{n-2s}}\leqslant\mathcal{Z}(x) \leqslant \frac{c_1}{|x|^{n-2s}} \quad \text{if } |x| > 1, \] and
\[\frac{c_4}{|x|^{n-2}}\leqslant\mathcal{Z}(x)\leqslant \frac{c_2}{|x|^{n-2}}\quad \text{if } 0< |x| \leqslant1,
\]
for some constants $c_1,$ $c_2$, \( c_3 \) and \( c_4 \) depending only on \( n \) and \( s \).

\end{theorem}
\begin{lemma}
Let $n> 2$ and \(\Phi : \mathbb{R}^n \to \mathbb{C}\) be in the Schwartz space. Then, the function  
\[\mathbb{R}^n \times (0, +\infty) \ni (\xi, t) \mapsto e^{-t(|\xi|^{2s} + |\xi|^2 )} \Phi(\xi)\]
belongs to \(L^1(\mathbb{R}^n \times (0, +\infty), \mathbb{C})\) and  
\begin{equation}\label{equ:4.6}\iint_{\mathbb{R}^n \times (0, +\infty)} e^{-t(|\xi|^{2s} + |\xi|^2 )} \Phi(\xi)  d\xi  dt = \int_{\mathbb{R}^n} \frac{\Phi(\xi)}{ |\xi|^{2s} + |\xi|^2}  d\xi. \end{equation}
\end{lemma}

\begin{proof}
We have that  
\[
\int_0^{+\infty} e^{-t(|\xi|^{2s} + |\xi|^2 )}  dt = \frac{1}{ |\xi|^{2s} + |\xi|^2}.
\]
Therefore, by the Fubini-Tonelli's Theorem,
\begin{align*}
&\iint_{\mathbb{R}^n \times (0, +\infty)} |e^{-t(|\xi|^{2s}+|\xi|^2)} \Phi(\xi)|  d\xi  dt \\
&= \iint_{\mathbb{R}^n \times (0, +\infty)} e^{-t(|\xi|^{2s}+|\xi|^2)} |\Phi(\xi)|  d\xi  dt \\
&= \int_{\mathbb{R}^n} \frac{|\Phi(\xi)|}{|\xi|^{2s}+|\xi|^2}  d\xi,
\end{align*}
which is finite. \qedhere
\end{proof}

With this preparatory work, we can now check that $\mathcal{Z}$ is the fundamental solution of the mixed order operator \(-\Delta +(-\Delta)^s\), as clarified by the following result.

\begin{lemma}\label{lem4.4} Let $n>2,$
the Fourier transform of \(\mathcal{Z}\) equals \(\frac{1}{|\xi|^{2s}+|\xi|^2}\) in the sense of distribution.

More explicitly, for every \(\phi : \mathbb{R}^n \to \mathbb{R}\) in the Schwartz space, we have that
\begin{equation}\label{lemm4.7}
\int_{\mathbb{R}^n} \mathcal{Z}(x) \phi(x)  dx = \int_{\mathbb{R}^n} \frac{\hat{\phi}(\xi)}{|\xi|^{2s}+|\xi|^2}  d\xi. 
\end{equation}
\end{lemma}

\begin{proof}
By \eqref{Bessel} and the Fubini-Tonelli's Theorem (whose validity is a consequence of Lemma~\ref{lem4.1}), we have that
\begin{align*}
\int_{\mathbb{R}^n} \mathcal{Z}(x) \phi(x)  dx &= \int_{\mathbb{R}^n} \left( \int_0^{+\infty}   \mathcal{H}(x,t)  dt \right) \phi(x)  dx \\
&= \int_0^{+\infty} \left( \int_{\mathbb{R}^n}   \mathcal{H}(x,t) \phi(x)  dx \right)  dt.
\end{align*}
Hence, by \eqref{heat kernel},
\[
\int_{\mathbb{R}^n} \mathcal{Z}(x)  \phi(x)  dx = \int_0^{+\infty} \left( \int_{\mathbb{R}^n} \left( \int_{\mathbb{R}^n} e^{-t(|\xi|^{2s}+|\xi|^2)+2\pi i x \cdot \xi}  \phi(x)  d\xi \right)  dx \right)  dt.
\]
This and the Fubini-Tonelli's Theorem (whose validity is a consequence of \cite[Lemma B.4]{SerenaDipierro2025}) yield that
\begin{align*}
&\int_{\mathbb{R}^n} \mathcal{Z}(x)  \phi(x)  dx \\
&= \int_0^{+\infty} \left( \int_{\mathbb{R}^n} \left( \int_{\mathbb{R}^n} e^{-t(|\xi|^{2s}+|\xi|^2)+2\pi i x \cdot \xi}  \phi(x)  dx \right)  d\xi \right)  dt \\
&= \int_0^{+\infty} \left( \int_{\mathbb{R}^n} e^{-t(|\xi|^{2s}+|\xi|^2)}  \hat{\phi}(\xi)  d\xi \right)  dt,
\end{align*}
where $\hat{\phi}(\xi)=\int_{\mathbb{R}^n}e^{2\pi ix\cdot\xi}\phi(x)dx$ is the Fourier transform of $\phi.$
From this and \eqref{equ:4.6} (utilized here with \(\Phi := \hat{\phi}\)) we arrive at \eqref{lemm4.7}, as desired. 
\end{proof}
}}
\begin{lemma} \label{proposition}
Let \(t=\frac{2n}{n+2}\). If $f\in L^t(\mathbb R^n)\cap L^\infty(\mathbb R^n),$
then $u(x)=(\mathcal{Z}*f)(x)$ is the unique weak solution of
\[
-\Delta u+(-\Delta)^s u=f\quad\text{in }\mathbb R^n.
\] \end{lemma}
\begin{proof} Fix $r>\max\left\{t,\frac n2\right\}.$
Since \(f\in L^t(\mathbb R^n)\cap L^\infty(\mathbb R^n)\), we have \(f\in L^r(\mathbb R^n)\). Let $\mathcal{Z}_0=\mathcal{Z}\mathbf1_{B_1}, \, \mathcal{Z}_{\infty}=\mathcal{Z}\mathbf1_{\mathbb{R}^n\setminus B_1}.$
By Theorem \ref{kernelz}, $(n-2)r'<n \text{ and } (n-2s)t'>n,$ we have
\[
\mathcal{Z}_0\in L^{r'}(\mathbb R^n),
\qquad
\mathcal{Z}_{\infty}\in L^{t'}(\mathbb R^n),
\]
Consequently, for any $g\in L^t\cap L^r,$ the H\"older’s inequality yields
\begin{equation}\label{*}
\|\mathcal{Z}*g\|_{L^{\infty}(\mathbb{R}^n)}
\leqslant C\bigl(\|g\|_{L^t}+\|g\|_{L^r}\bigr)
\end{equation}
for every \(g\in L^t(\mathbb R^n)\cap L^r(\mathbb R^n)\). 

Choose \(f_j\in\mathcal S(\mathbb R^n)\) such that
\begin{equation}\label{**}
\widehat f_j\in C_c^\infty(\mathbb R^n\setminus\{0\}),
\qquad
f_j\rightarrow f
\quad\text{in }L^t(\mathbb R^n)\cap L^r(\mathbb R^n).
\end{equation}
Set
\[
u_j:=\mathcal{Z}*f_j.
\]
By the Lemma \ref{lem4.4}, we have
\[
\widehat u_j=\frac{\widehat f_j}{|\xi|^2+|\xi|^{2s}},
\]
and
\begin{equation}
u_j\in\mathcal S(\mathbb R^n)\subset \mathcal{X}^{1,2}(\mathbb R^n),
\quad -\Delta u_j+(-\Delta)^s u_j=f_j.
\end{equation}
For any $\varphi\in\mathcal{X}^{1,2}(\mathbb{R}^n),$ we have 
\[
\begin{aligned}
\int_{\mathbb R^n}\nabla (u_k-u_j)\cdot\nabla \varphi\,dx+\frac{c_{n,s}}2
\iint_{\mathbb R^{2n}}
\frac{\left((u_k-u_j)(x)-(u_k-u_j)(y)\right)(\varphi(x)-\varphi(y))}
{|x-y|^{n+2s}}\,dx\,dy=\int_{\mathbb{R}^n}(f_k-f_j)\varphi \,dx,
\end{aligned}
\]
Take $\varphi=u_k-u_j,$ we have 
\[
\begin{aligned}
&\int_{\mathbb R^n}|\nabla(u_k-u_j)|^2\,dx+\frac{c_{n,s}}2
\iint_{\mathbb R^{2n}}
\frac{|(u_j-u_k)(x)-(u_j-u_k)(y)|^2}
{|x-y|^{n+2s}}\,dx\,dy\\
&=\int_{\mathbb R^n}(f_j-f_k)(u_j-u_k)\,dx.
\end{aligned}
\]
Since \(t'=2^*\), the Hölder’s inequality and the Sobolev inequality imply
\[
\|u_j-u_k\|
\leqslant C\|f_j-f_k\|_{L^t}.
\]
Thus \(\{u_j\}\) is a Cauchy sequence in \(\mathcal{X}^{1,2}(\mathbb R^n)\). By completeness, there exists \(u\in X^{1,2}(\mathbb R^n)\) such that
\[
u_j\rightarrow u
\quad\text{in }\mathcal{X}^{1,2}(\mathbb R^n)
\quad\text{and in }L^{2^*}(\mathbb R^n).
\]
On the other hand, \eqref{*} and \eqref{**} give
\[
\|u_j-\mathcal{Z}*f\|_\infty
\leqslant C\bigl(
\|f_j-f\|_{L^t}+\|f_j-f\|_{L^r}
\bigr)\longrightarrow0.
\]
Comparing this uniform limit with an almost everywhere convergent subsequence of \(\{u_j\}\), we obtain
\[
u=\mathcal{Z}*f\quad\text{a.e. in }\mathbb R^n.
\]
Passing to the limit in the weak formulation for \(u_j\), using convergence in \(\mathcal{X}^{1,2}(\mathbb{R}^n)\) and \(f_j\to f\) in \(L^t\), shows that \(u=\mathcal{Z}*f\) is a weak solution of \(-\Delta u+(-\Delta)^s u=f\).
Finally, if \(w\in \mathcal{X}^{1,2}(\mathbb R^n)\) is another weak solution, subtracting the two weak formulations and testing with \(u-w\) yields
\[
\int_{\mathbb R^n}|\nabla(u-w)|^2\,dx
+\frac{c_{n,s}}2
\iint_{\mathbb R^{2n}}
\frac{|(u-w)(x)-(u-w)(y)|^2}
{|x-y|^{n+2s}}\,dx\,dy=0.
\]
The Sobolev inequality then gives \(u=w\) almost everywhere. This proves uniqueness and completes the proof.
\end{proof}

	\section{Regularity theory  for viscosity solutions }
	In this section, we mainly address the regularity of weak solutions of  \eqref{equ1} (also shown to be viscosity solutions). In Sections \ref{sec3} and \ref{sec4}, we first obtain the $C^{1,\alpha}$ and $C^{2,\alpha}$-regularity of weak solutions to equation \eqref{equ1}. Based on this regularity result, in Section \ref{sec:existence of vs}, we prove that these weak solutions  are also viscosity solutions.
	
	\subsection{The regularity of weak solutions for $p\in(0,1)$}\label{sec3}
	In this section we prove a H\"older estimate of solution to the equation \eqref{equ1} when $p\in(0,1).$ Our proof is based on the arguments developed in \cite{MR2735074,MR2863859}, similar H\"older estimates with very different proofs can be found in \cite{MR2494809} for symmetric kernels and in \cite{MR2974279} for non-symmetric kernels.
	
	\begin{lemma} \label{lem:4.2l}
Let $s \in (0,1)$, $\alpha \in (0, \min\{1, 2s\})$, and define $\phi(x) = C_1 |x|^\alpha$. Let
\[
Q= \int_{\mathcal{C}} \frac{\phi(a + 2z) + \phi(a - 2z) - 2\phi(a)}{|z|^{n+2s}} dz,
\]
and
\[
\mathcal{C} = \left\{ z \in \mathbb{R}^n : |z| < \eta_1 |a|,\ |z \cdot a| \geqslant (1 - \eta_2)|a||z| \right\},
\]
with $\eta_1, \eta_2 \in (0, 1/2)$ being constants depending only on $\alpha$.

Then there exists a constant $N = N(n, \alpha) > 0$ such that
\[
Q \leqslant -N C_1 |a|^{\alpha - 2s}.
\]
\end{lemma}

\begin{proof}
Define $
\varphi(t):= \phi(a + 2tz) = C_1 |a + 2tz|^\alpha.
$
A straightforward calculation yields
\[
\varphi'(t) = 2C_1 \alpha |a + 2tz|^{\alpha - 2} (a + 2tz) \cdot z,
\]
\[
\varphi''(t) = 4C_1 \alpha \left[ (\alpha - 2) |a + 2tz|^{\alpha - 4} \left| (a + 2tz) \cdot z \right|^2 + |a + 2tz|^{\alpha - 2} |z|^2 \right].
\]

On the set $\mathcal{C}$, we have the estimates, for all $t\in[-1,1],$
\begin{align*}
|a + 2tz| &\leqslant (1 + 2\eta_1)|a|, \\
|(a + 2tz) \cdot z| &\geqslant (1 - \eta_2)|a||z| - 2|z|^2 \geqslant (1 - 2\eta_1 - \eta_2)|a||z|.
\end{align*}
Since $\alpha - 2 < 0$, we obtain:
\begin{equation}\label{star}
\varphi''(t) \leqslant 4C_1 \alpha |a + 2tz|^{\alpha - 4} \left[ (\alpha - 2)(1 - 2\eta_1 - \eta_2)^2 + (1 + 2\eta_1)^2 \right] |a|^2 |z|^2.
\end{equation}

By choosing $\eta_1, \eta_2\in(0,\frac{1}{2})$ sufficiently small (depending only on $\alpha\in(0,1)$) such that
\[
(\alpha - 2)(1 - 2\eta_1 - \eta_2)^2 + (1 + 2\eta_1)^2 \leqslant \frac{\alpha - 1}{2} < 0.
\]
This together with \eqref{star} implies that
\[
\varphi''(t) \leqslant -2C_1 \alpha (1 - \alpha) |a + 2tz|^{\alpha - 4} |a|^2 |z|^2.
\]
Using the bound $|a + 2tz|^{\alpha - 4} \geqslant 2^{\alpha - 4} |a|^{\alpha - 4}$, we further obtain,
\[
\varphi''(t) \leqslant -2^{\alpha - 3} C_1 \alpha (1 - \alpha) |a|^{\alpha - 2} |z|^2\quad t\in[-1,1],\quad z\in\mathcal{C}.
\]
By the mean value theorem, there exists $t_0 \in (-1, 1)$ such that
\[
\phi(a + 2z) + \phi(a - 2z) - 2\phi(a) = \varphi(1) + \varphi(-1) - 2\varphi(0) = \varphi''(t_0).
\]
Therefore,
\begin{equation*}
\begin{aligned}
Q &\leqslant -\int_{\mathcal{C}}\frac{ 2^{\alpha - 3} C_1 \alpha (1 - \alpha) |a|^{\alpha - 2} |z|^2}{|z|^{n+2s}}  dz=-2^{\alpha-3}C_1\alpha(1-\alpha)|a|^{\alpha-2}\int_{\mathcal{C}}\frac{|z|^2}{|z|^{n+2s}}dz\\
&\leqslant-NC_1|a|^{\alpha-2s},
\end{aligned}
\end{equation*}
where $N=N(\alpha,n)>0.$
This completes the proof.
\end{proof}

	\begin{lemma}[$C^\alpha$-estimate] \label{thm:4.1}
Let  $0 < s < 1$, $1/2 \leqslant r < R < 1$, and $f \in L^\infty(B_1)$. 
Let $u \in C_{loc}^2(B_1) \cap L_1(\mathbb{R}^n, \omega)$ with $\omega(x) = \frac{1}{1 + |x|^{n+2s}}$ such that
\[
\mathcal{L}u:= f \quad \text{in } B_R.
\]
Then for any $\alpha \in (0, \min\{1, 2s\})$, we have
\begin{align*}
&[u]_{C^\alpha(B_r)}\\ &\leqslant N \left( (R - r)^{-\alpha} \sup_{B_R} |u|+ (R - r)^{2s-\alpha-2} \sup_{B_R} |u| + (R - r)^{-n-\alpha} \|u\|_{L_1(\mathbb{R}^n, \omega)} + (R - r)^{2s - \alpha} \text{osc}_{B_R} f \right),
\end{align*}
where $N = N(n, s, \alpha)$.
\end{lemma}

\begin{proof}\textbf{1. Localization problem:}
Denote $r_1 = (R - r)/2$ and $\bar{r} = (R + r)/2$. Set $w(x) = \mathds{1}_{B_R}(x) u(x)$. If $\|u\|_{L^{\infty}(B_R)}=0,$ the desired estimate is immediate. We therefore assume that  $\|u\|_{L^{\infty}(B_R)}>0.$

 For $x \in B_{\bar{r}}$, we have $\Delta u(x) = \Delta w(x)$ and
\begin{align*}
(-\Delta)^s u(x) &= C_{n,s} P.V.\int_{\mathbb{R}^n}\frac{u(x)-u(x+z)}{|z|^{n+2s}} dz \\
&=C_{n,s} P.V. \int_{|z|<r_1}\frac{u(x)-u(x+z)}{|z|^{n+2s}} dz + \int_{|z|\geqslant r_1}\frac{u(x)-u(x+z)}{|z|^{n+2s}} dz\\
&= C_{n,s} P.V. \int_{|z|<r_1}\frac{w(x)-w(x+z)}{|z|^{n+2s}} dz + \int_{|z|\geqslant r_1}\frac{w(x)-u(x+z)}{|z|^{n+2s}} dz\\
&=C_{n,s} P.V. \int_{|z|<r_1}\frac{w(x)-w(x+z)}{|z|^{n+2s}} dz+\int_{|z|\geqslant r_1}\frac{w(x)-w(x+z)}{|z|^{n+2s}} dz \\
&\quad+ \int_{|z|\geqslant r_1}\frac{w(x+z)-u(x+z)}{|z|^{n+2s}} dz=(-\Delta)^sw-\int_{|z|\geqslant r_1}\frac{u(x+z)-w(x+z)}{|z|^{n+2s}} dz.
\end{align*}
Hence for $x\in B_{\bar{r}}$,
\[
 - \mathcal{L}w(x) = g(x) - f(x),
\]
where
\[
g(x) = -C_{n,s}\int_{|z| \geqslant r_1} \frac{u(x+z)-w(x+z)}{|z|^{n+2s}} dz.
\]
Since $u-w$ vanishes in $B_R$ and $|x|<1,$ we obtain 
\begin{equation}\label{eq:2.1.}
\|g\|_{L^\infty(B_R)} \leqslant N r_1^{-n-2s} \|u\|_{L_1(\mathbb{R}^n, \omega)}, 
\end{equation}
where $N = N(n,s)$.

\textbf{2. Construct comparison functions:} For $x_0 \in B_r$, set
\[
M(x, y) := w(x) - w(y) - \phi(x - y) - \Gamma(x),
\]
where $\phi(z) = C_1 |z|^{\alpha}$, $\alpha \in (0, \min\{1, 2s\})$, and $\Gamma(x) = C_2 |x - x_0|^2$. We will find $C_1, C_2 \in (0, \infty)$ depending only on $n, s, r_1, \|u\|_{L^\infty(B_R)}, \|u\|_{L_1(\mathbb{R}^n, \omega)}, \text{osc}_{B_R} f,$  but independent of the choice of $x_0 \in B_r$, such that
\begin{equation}\label{eq:2.1}
\sup_{x, y \in \mathbb{R}^n} M(x, y) \leqslant 0. \end{equation}

To prove \eqref{eq:2.1}, we first take
\[
C_2 := 8 r_1^{-2} \|u\|_{L^\infty(B_R)}.
\]
Then, for $x \in \mathbb{R}^n \setminus B_{r_1/2}(x_0)$,
\[
w(x) - w(y) \leqslant 2 \|u\|_{L^\infty(B_R)} \leqslant C_2 |x - x_0|^2.
\]
This shows that
\begin{equation}\label{eq:2.2}
M(x, y) \leqslant 0, \quad x \in \mathbb{R}^n \setminus B_{r_1/2}(x_0). 
\end{equation}

To get a contradiction, we assume that there exist $x, y \in \mathbb{R}^n$ such that $M(x, y) > 0$. By \eqref{eq:2.2} we know that $x \in B_{r_1/2}(x_0) \subset B_{(\bar{r}+r)/2}$. Moreover, if $M(x, y) > 0$, then
\begin{equation}\label{eq:2.3}
w(x) - w(y) > C_1 |x - y|^{\alpha}, \quad \text{i.e.,} \quad |x - y|^{\alpha} < \frac{2 \|u\|_{L^\infty(B_R)}}{C_1}. 
\end{equation}
If we take a sufficiently large $C_1$ so that $C_1 \geqslant 2^{1+\alpha} r_1^{-\alpha} \|u\|_{L^\infty(B_R)}$, the above inequalities show that $y \in B_{\bar{r}}$.

Therefore, the assumption that $M(x, y) > 0$ for some $x, y \in \mathbb{R}^n$ (and the continuity of $u$ on $B_R$) enables us to assume that there exist $\bar{x}, \bar{y} \in B_{\bar{r}}$ satisfying $\sup_{x,y \in \mathbb{R}^n} M(x, y) = M(\bar{x}, \bar{y}) > 0$.

\textbf{3. Construct a contradiction:} Note that at $\bar{x}, \bar{y} \in B_{\bar{r}}$ we have
\begin{align*}
g(\bar{y}) - f(\bar{y}) &=  - \mathcal{L}w(\bar{y}), \\
-g(\bar{x}) + f(\bar{x}) &=  \mathcal{L}w(\bar{x}).
\end{align*}
Thus, it follows that
\begin{equation}\label{eq:2.5}
 \mathcal{L}w(\bar{x}) - \mathcal{L}w(\bar{y})\leqslant 2 \|g\|_{{L}^\infty(B_R)} +\text{osc}_{B_R} f. 
\end{equation}

Set \[
\mathcal{D}(\bar{x}, \bar{y}, z) = w(\bar{x} + z) + w(\bar{x} - z) - 2w(\bar{x}) - w(\bar{y} + z) - w(\bar{y} - z) + 2w(\bar{y}).
\] 
Set $a=\bar{x}-\bar{y},$ since $M(x, y)$ attains its maximum at $\bar{x}, \bar{y}$, the function $z\longmapsto M(\overline{x}+z,\overline{y}+z)$ attains its maximum at $z=0.$ Then
\begin{equation}\label{eq:2.61}
\begin{aligned}
-\Delta w(\overline{x})+ \Delta w(\bar{y})\geqslant -2nC_2,
 \end{aligned}
\end{equation}
and 
\begin{align*}
w(\bar{x} + z) - w(\bar{y} + z) - \phi(\bar{x} - \bar{y}) - \Gamma(\bar{x} + z) &\leqslant w(\bar{x}) - w(\bar{y}) - \phi(\bar{x} - \bar{y}) - \Gamma(\bar{x}), \\
w(\bar{x} - z) - w(\bar{y} - z) - \phi(\bar{x} - \bar{y}) - \Gamma(\bar{x} - z) &\leqslant w(\bar{x}) - w(\bar{y}) - \phi(\bar{x} - \bar{y}) - \Gamma(\bar{x})
\end{align*}
for all $z \in \mathbb{R}^n$. These two inequalities lead us to
\begin{equation}\label{eq:2.4}
\mathcal{D}(\bar{x}, \bar{y}, z) \leqslant \Gamma(\bar{x} + z) + \Gamma(\bar{x} - z) - 2\Gamma(\bar{x}), \quad z \in \mathbb{R}^n. 
\end{equation}
By again the assumption that $M(x, y)$ has the maximum at $\bar{x}, \bar{y}$, we have
\begin{align*}
w(\bar{x} + z) - w(\bar{y} - z) - \phi(\bar{x} - \bar{y} + 2z) - \Gamma(\bar{x} + z) &\leqslant w(\bar{x}) - w(\bar{y}) - \phi(\bar{x} - \bar{y}) - \Gamma(\bar{x}), \\
w(\bar{x} - z) - w(\bar{y} + z) - \phi(\bar{x} - \bar{y} - 2z) - \Gamma(\bar{x} - z) &\leqslant w(\bar{x}) - w(\bar{y}) - \phi(\bar{x} - \bar{y}) - \Gamma(\bar{x})
\end{align*}
for all $z \in \mathbb{R}^n$. Hence it follows that, for any $z \in \mathbb{R}^n$,
\begin{equation}\label{eq:2.8}
\mathcal{D}(\bar{x}, \bar{y}, z) \leqslant \phi(\bar{x} - \bar{y} + 2z) + \phi(\bar{x} - \bar{y} - 2z) - 2\phi(\bar{x} - \bar{y}) + \Gamma(\bar{x} + z) + \Gamma(\bar{x} - z) - 2\Gamma(\bar{x}). 
\end{equation}

Since $\bar{x}, \bar{y}$ satisfy \eqref{eq:2.3}, we have $|a| < r_1/2$. Also set, for some $\eta_1, \eta_2 \in (0, 1/2)$,
\[
\mathcal{C} = \{ |z| < \eta_1 |a| : |z \cdot a| \geqslant (1 - \eta_2)|a||z| \}.
\]
Then $\mathcal{C} \subset B_{r_1}$ and
\begin{equation}\label{eq:2.6}
-2J=c_{n,s} \int_{|z| \geqslant r_1} \frac{\mathcal{D}(\bar{x}, \bar{y}, z)}{|z|^{n+2s}} dz +c_{n,s}\int_{B_{r_1} \setminus \mathcal{C}} \frac{\mathcal{D}(\bar{x}, \bar{y}, z)}{|z|^{n+2s}} dz + c_{n,s}\int_{\mathcal{C}} \frac{\mathcal{D}(\bar{x}, \bar{y}, z)}{|z|^{n+2s}} dz := c_{n,s}(J_1 + J_2 + J_3),
\end{equation}
where $J:=(-\Delta)^s w(\bar{x})-(-\Delta)^s w(\bar{y}).$

Note that
\[
J_1 \leqslant N(n,s) r_1^{-2s} \|u\|_{L^\infty(B_R)}.
\]
By \eqref{eq:2.4} it follows
\[
J_2 \leqslant \int_{B_{r_1} \setminus \mathcal{C}} \frac{\Gamma(\bar{x} + z) + \Gamma(\bar{x} - z) - 2\Gamma(\bar{x})}{|z|^{n+2s}}  dz \leqslant N(n,s) r_1^{2-2s} C_2,
\]
where  $N$ is independent of $\eta_1, \eta_2$ in the definition of $\mathcal{C}$.

Now using \eqref{eq:2.8} we obtain
\begin{equation*}
\begin{aligned}
J_3 &\leqslant \int_{\mathcal{C}} \frac{\phi(\bar{x} - \bar{y} + 2z) + \phi(\bar{x} - \bar{y} - 2z) - 2\phi(\bar{x} - \bar{y})}{|z|^{n+2s}} dz + \int_{\mathcal{C}} \frac{\Gamma(\bar{x} + z) + \Gamma(\bar{x} - z) - 2\Gamma(\bar{x})}{|z|^{n+2s}} \\
&:= J_{3,1} + J_{3,2}.
\end{aligned}
\end{equation*}
The term $J_{3,2}$ is again bounded by $N(n,s) r_1^{2-2s} C_2$. Finally, by Lemma \ref{lem:4.2l},
\[
J_{3,1} \leqslant -N(n, \alpha) C_1 |a|^{\alpha-2s}.
\]
Thus, we get from \eqref{eq:2.6} and the choice of $C_2$ that
\begin{equation}\label{eq:2.10}
J \geqslant -N(n, s) r_1^{-2s} \|u\|_{L^\infty(B_R)} + N(n, \alpha) C_1 |a|^{\alpha-2s}. 
\end{equation}
Combining \eqref{eq:2.1.}, \eqref{eq:2.5}, \eqref{eq:2.61}, \eqref{eq:2.6} and \eqref{eq:2.10} we finally have
\begin{equation*}
\begin{aligned}
N(n,\alpha)C_1|a|^{\alpha-2s}
\leqslant N(n,s)\Bigl(
&\operatorname{osc}_{B_R}f
+r_1^{-2s}\|u\|_{L^\infty(B_R)}\\
&+r_1^{-2}\|u\|_{L^\infty(B_R)}
+r_1^{-n-2s}\|u\|_{L^1(\mathbb R^n,\omega)}
\Bigr).
\end{aligned}
\end{equation*}
Choose 
\[
\begin{aligned}
C_1={}&
2^{1+\alpha}r_1^{-\alpha}\|u\|_{L^\infty(B_R)}\\
&+\frac{N(n,s)}{N(n,\alpha)}
r_1^{2s-\alpha}
\Bigl(
\operatorname{osc}_{B_R}f
+r_1^{-2s}\|u\|_{L^\infty(B_R)}
+r_1^{-2}\|u\|_{L^\infty(B_R)}
+r_1^{-n-2s}\|u\|_{L^1(\mathbb R^n,\omega)}
\Bigr).
\end{aligned}
\]
Then, for \(\alpha \in (0, \min\{1, 2s\})\), by \eqref{eq:2.3} \(|a|^{\alpha - 2s} r_1^{2s - \alpha} > 1\) and
\begin{equation*}
\begin{aligned}
C_1>\frac{N(n,s)}{N(n,\alpha)}
|a|^{2s-\alpha}
\Bigl(
\operatorname{osc}_{B_R}f
+r_1^{-2s}\|u\|_{L^\infty(B_R)}
+r_1^{-2}\|u\|_{L^\infty(B_R)}
+r_1^{-n-2s}\|u\|_{L^1(\mathbb R^n,\omega)}
\Bigr).
\end{aligned}
\end{equation*}
This is a contradiction.

This proves the assertion in the theorem. More specifically, using the fact that $C_1$ and $C_2$ are independent of the choice of $x_0 \in B_r$, we obtain
\[
|u(x) - u(y)| \leqslant C_1 |x - y|^{\alpha}, \quad x, y \in B_r.
\]
Expanding the chosen expression for \(C_1\), we find
\[
\begin{aligned}
[u]_{C^\alpha(B_r)}
\leqslant N(n,s,\alpha)\Bigl(
&r_1^{-\alpha}\|u\|_{L^\infty(B_R)}
+r_1^{2s-\alpha-2}\|u\|_{L^\infty(B_R)}\\
&+r_1^{-n-\alpha}\|u\|_{L^1(\mathbb R^n,\omega)}
+r_1^{2s-\alpha}\operatorname{osc}_{B_R}f
\Bigr).
\end{aligned}
\]Substituting \(r_1=(R-r)/2\) and absorbing the resulting powers of \(2\) into the constant gives the desired estimate.
\end{proof}

	\begin{corollary} \label{cor:4.3}
Let  $0 < s < 1$, and $f \in L^\infty(B_1)$. Let $u \in C^2_{\mathrm{loc}}(B_1) \cap L_1(\mathbb{R}^n, \omega)$ with $\omega(x) = 1/(1 + |x|^{n+2s})$ such that
\[
\mathcal{L}u  = f
\]
in $B_1$. Then for any $\alpha \in (0, \min\{1, 2s\})$, we have
\begin{equation}\label{equ:2.11}
[u]_{C^\alpha(B_{1/2})} \leqslant N\|u\|_{L_1(\mathbb{R}^n, \omega)} + N \text{osc}_{B_1} f,
\end{equation}
where $N = N(n,s, \alpha)$.
\end{corollary}

\begin{proof}
Set
\[
r_k = \frac{3}{4} - 2^{-k-2}, \quad B_{(k)} = B_{r_k}, \quad k = 0, 1, 2, \ldots.
\]
Lemma \ref{thm:4.1} gives, for $k = 0, 1, 2, \ldots$,
\begin{equation}\label{eq:2.11}
[u]_{C^\alpha(B_{(k)})} \leqslant N_1 \left( 2^{2k} \sup_{B_{(k+1)}} |u| + 2^{(n+\alpha)k} \|u\|_{L_1(\mathbb{R}^n, \omega)} + \text{osc}_{B_{1}} f \right),
\end{equation}
where $N_1 = N_1(n,s, \alpha)$ is a constant independent of $k$. To estimate the first term on the right-hand side of \eqref{eq:2.11}, by the well-known interpolation inequality, we have
\begin{equation}\label{eq:2.12}
\sup_{B_{(k+1)}} |u| \leqslant \varepsilon [u]_{C^\alpha(B_{(k+1)})} + N \varepsilon^{-n/\alpha} \|u\|_{L_1(B_{(k+1)})}, \quad \forall \epsilon \in (0, 1). 
\end{equation}
Multiplying both sides of Equation \eqref{eq:2.12} by $N_12^{2k}$, and upon taking $\epsilon = (N_1 2^{2k+3n/\alpha})^{-1},$ we have
\begin{equation*}
\begin{aligned}
N_12^{2k}\sup_{B_{(k+1)}} |u| &\leqslant 2^{2k}N_1\varepsilon [u]_{C^\alpha(B_{(k+1)})} + NN_12^{2k} \varepsilon^{-n/\alpha} \|u\|_{L_1(B_{(k+1)})}\\
&= \frac{1}{2^{\frac{3n}{\alpha}}}[u]_{C^\alpha(B_{(k+1)})}+NN_1^{1+\frac{n}{\alpha}}2^{2k(1+\frac{n}{\alpha})}2^{\frac{3n^2}{\alpha^2}}\|u\|_{L_1(B_{(k+1)})}\\
&\leqslant \frac{1}{2^{\frac{3n}{\alpha}}}[u]_{C^\alpha(B_{(k+1)})}+N(\alpha,s,n)2^{2k(1+\frac{n}{\alpha})}\|u\|_{L_1(B_1)}.
\end{aligned}
\end{equation*}
Substituting the above inequality into \eqref{eq:2.11} yields
\[
\begin{aligned}
[u]_{C^\alpha(B_{(k)})}
\le{}&
2^{-3n/\alpha}[u]_{C^\alpha(B_{(k+1)})}+N\,2^{2k(1+n/\alpha)}\|u\|_{L_1(B_1)}+N_1\,2^{(n+\alpha)k}\|u\|_{L_1(\mathbb R^n,\omega)}
+N_1\operatorname{osc}_{B_1}f,
\end{aligned}
\]
and multiplying both sides of the resulting inequality by  $2^{-3kn/\alpha}$ gives
\begin{equation*}
\begin{aligned}
&\quad2^{-\frac{3kn}{\alpha}}[u]_{C^\alpha(B_{(k)})}\\
& \leqslant 2^{-\frac{3n(k+1)}{\alpha}}[u]_{C^\alpha(B_{(k+1)})}+N 2^{-k(\frac{n}{\alpha}-2)} \|u\|_{L_1(B_1)} +N_1 2^{-k(\frac{3n}{\alpha}-n-\alpha)} \|u\|_{L_1(\mathbb{R}^n, \omega)} + N_1 2^{-\frac{3kn}{\alpha}}\text{osc}_{B_{1}} f.
\end{aligned}
\end{equation*}
Let $A_k=2^{-\frac{3kn}{\alpha}}[u]_{C^\alpha(B_{(k)})},$ then sum over $k$ to obtain
\begin{align*}
 A_k
&\leqslant  A_{k+1} + N 2^{-k(\frac{n}{\alpha}-2)} \|u\|_{L_1(B_1)} +N_1 2^{-k(\frac{3n}{\alpha}-n-\alpha)} \|u\|_{L_1(\mathbb{R}^n, \omega)} + N_1 2^{-\frac{3kn}{\alpha}}\text{osc}_{B_{1}} f.
\end{align*}
Summing this inequality over \(k=0,\ldots,m\), we obtain
\[
\begin{aligned}
A_0\leqslant A_{m+1}
&+N\|u\|_{L_1(B_1)}
\sum_{k=0}^{m}2^{-k(n/\alpha-2)}+N_1\|u\|_{L_1(\mathbb R^n,\omega)}
\sum_{k=0}^{m}2^{-k(3n/\alpha-n-\alpha)}+N_1\operatorname{osc}_{B_1}f
\sum_{k=0}^{m}2^{-3kn/\alpha}.
\end{aligned}
\]Under the standing assumption \(n\geqslant2\), and since \(0<\alpha<1\), we have
\[
\frac n\alpha-2>0,
\qquad
\frac{3n}{\alpha}-n-\alpha>0.
\]Hence all three geometric series converge.
Furthermore, \(u\in C^2_{\mathrm{loc}}(B_1)\) implies
\[
[u]_{C^\alpha(B_{3/4})}<\infty.
\]Since \(B_{(m+1)}\subset B_{3/4}\), it follows that
\[
0\leqslant A_{m+1}
\leqslant
2^{-3(m+1)n/\alpha}[u]_{C^\alpha(B_{3/4})}
\longrightarrow0.
\]We may therefore let \(m\to\infty\) and conclude that
\[
A_0
\leqslant N\|u\|_{L_1(B_1)}
+N_1\|u\|_{L_1(\mathbb R^n,\omega)}
+N_1\operatorname{osc}_{B_1}f.
\]Finally, \(B_{(0)}=B_{1/2}\), so
\[
A_0=[u]_{C^\alpha(B_{1/2})}.
\]Moreover, \(\omega(x)\geqslant1/2\) on \(B_1\), and hence
\[
\|u\|_{L_1(B_1)}
\le2\|u\|_{L_1(\mathbb R^n,\omega)}.
\]Combining these observations gives \eqref{equ:2.11}, where  constant $N$ depends only on \(n,s,\alpha\), and is independent of \(u\) and \(f\). This completes the proof. \qedhere
\end{proof}
	
	\begin{proposition}\label{pro:2.4} Let \( f \in L^\infty(B_1) \) and let \( u \in \mathcal{X}^{1,2}(B_3) \cap L^\infty(\mathbb{R}^n) \) be a weak solution of
\begin{equation}\label{equ:2.1}\mathcal{L} u = f \quad \text{in} \quad B_1.\end{equation} Then, \( u \in C^\alpha(B_{1/4}) \) for any \( 0 < \alpha < \min\{2s, 1\} \) and \[\|u\|_{C^\alpha(B_{1/4})} \leqslant C \left( \|u\|_{L^\infty(\mathbb{R}^n)} + \|f\|_{L^\infty(B_1)}\right),\] for a suitable \( C > 0 \) depending on \( n \), \( s \) and \( \alpha \).
\end{proposition}
\begin{proof}
We take \(\rho \in C_0^\infty(B_1;[0,1])\) and we consider the mollifier \(\rho_\varepsilon(x) := \varepsilon^{-n} \rho(x/\varepsilon)\), \(\varepsilon > 0\). We define \(u_\varepsilon := u * \rho_\varepsilon\) and \(f_\varepsilon := f * \rho_\varepsilon\).

Of course, \(u_\varepsilon\) is a smooth function, by construction. Let us show that \(u_\varepsilon\) solves the following equation
\begin{equation} \label{eq:3.2}
-\Delta u_\varepsilon+(-\Delta)^s u_\varepsilon = f_\varepsilon \quad \text{in} \quad B_{1/2}.
\end{equation}
For this, first we observe that for any test function \(\phi\in C_0^{\infty}(B_{1/2})\), we have 
\begin{equation}
\begin{aligned}
\int_{\mathbb{R}^n}-\Delta u_{\epsilon}\phi dx&=\int_{\mathbb{R}^n}\triangledown u_{\epsilon}\cdot\triangledown\phi dx=\int_{\mathbb{R}^n}[(\triangledown u)*\rho_{\epsilon}]\cdot\triangledown\phi dx\\
&=\int_{\mathbb{R}^n}\left[\int_{\mathbb{R}^n}\triangledown u(x+z)\rho_{\epsilon}(z)\triangledown\phi(x)dx\right]dz\\
&=\int_{B_{\epsilon}}\left[\int_{\mathbb{R}^n}\triangledown u(\tilde{x})\rho_{\epsilon}(z)\triangledown\phi(\tilde{x}-z)d\tilde{x}\right]dz,
\end{aligned}
\end{equation}
and since \(u \in \mathcal{X}^{1,2}(B_3) \cap L^\infty(\mathbb{R}^n)\),
\begin{align*}
&\int_{\mathbb{R}^n} \left[ \int_{\mathbb{R}^{2n}} \frac{|u(x+z) - u(y+z)| |\phi(x) - \phi(y)| \rho_\varepsilon(z)}{|x-y|^{n+2s}}  dx  dy \right] dz \\
&= \int_{B_\varepsilon} \left[ \int_{Q(B_1)} \frac{|u(x+z) - u(y+z)| |\phi(x) - \phi(y)| \rho_\varepsilon(z)}{|x-y|^{n+2s}}  dx  dy \right] dz \\
&\leqslant \frac{\varepsilon^{-n}}{2} \int_{B_\varepsilon} \left[ \int_{Q(B_2)} \frac{|u(x) - u(y)|^2}{|x-y|^{n+2s}}  dx  dy + \int_{Q(B_1)} \frac{|\phi(x) - \phi(y)|^2}{|x-y|^{n+2s}}  dx  dy \right] dz \\
&< +\infty.
\end{align*}
 Here we used the notation
\[
Q(B_r) := \mathbb{R}^{2n} \setminus \left( (\mathbb{R}^n \setminus B_r) \times (\mathbb{R}^n \setminus B_r)\right), \quad r > 0.
\]
Hence, by Tonelli's Theorem the function
\[
\mathbb{R}^{2n} \times \mathbb{R}^n \ni (x, y, z) \mapsto \frac{(u(x + z) - u(y + z)) (\phi(x) - \phi(y)) \rho_\varepsilon(z)}{|x - y|^{n+2s}}
\]
belongs to \( L^1(\mathbb{R}^{2n} \times \mathbb{R}^n) \). Then, by Fubini's Theorem and the definition of $u_{\varepsilon},$
\begin{align*}
&\int_{\mathbb{R}^n} \left[ \int_{\mathbb{R}^{2n}} \frac{(u(x+z)-u(y+z)) (\phi(x)-\phi(y)) \rho_\varepsilon(z)}{|x-y|^{n+2s}}  dx  dy \right]  dz \\
&= \int_{\mathbb{R}^{2n}} \left[ \int_{\mathbb{R}^n} \frac{(u(x+z)-u(y+z)) (\phi(x)-\phi(y)) \rho_\varepsilon(z)}{|x-y|^{n+2s}}  dx  dy \right] \\
&= \int_{\mathbb{R}^{2n}} \frac{(u_\varepsilon(x)-u_\varepsilon(y)) (\phi(x)-\phi(y))}{|x-y|^{n+2s}}  dx dy.
\end{align*}
Therefore, we may use Fubini's Theorem and obtain that for any \(\phi \in C_0^\infty(B_{1/2}),\)
\begin{align*}
&\int_{\mathbb{R}^n} f_\varepsilon(x) \phi(x)  dx \\
&= \int_{\mathbb{R}^n} \left[ \int_{\mathbb{R}^n} f(x + z) \phi(x) \rho_\varepsilon(z)  dx \right]  dz \\
&= \int_{B_\varepsilon} \left[ \int_{\mathbb{R}^n} f(\tilde{x}) \phi(\tilde{x} - z) \rho_\varepsilon(z)  d\tilde{x} \right]  dz \\
&= \int_{B_\varepsilon} \left[ \int_{\mathbb{R}^{2n}} \frac{(u(\tilde{x}) - u(\tilde{y})) (\phi(\tilde{x} - z) - \phi(\tilde{y} - z)) \rho_\varepsilon(z)}{|\tilde{x} - \tilde{y}|^{n+2s}}  d\tilde{x}  d\tilde{y} \right]  dz +\int_{B_{\epsilon}}\left[\int_{\mathbb{R}^n}\triangledown u(\tilde{x})\rho_{\epsilon}(z)\triangledown\phi(\tilde{x}-z)d\tilde{x}\right]dz\\
&= \int_{\mathbb{R}^n} \left[ \int_{\mathbb{R}^{2n}} \frac{(u(x + z) - u(y + z)) (\phi(x) - \phi(y)) \rho_\varepsilon(z)}{|x - y|^{n+2s}}  dx  dy \right]  dz+\int_{\mathbb{R}^n}\left[\int_{\mathbb{R}^n}\triangledown u(x+z)\rho_{\epsilon}(z)\triangledown\phi(x)d\tilde{x}\right]dz \\
&= \int_{\mathbb{R}^{2n}} \frac{(u_\varepsilon(x) - u_\varepsilon(y)) (\phi(x) - \phi(y))}{|x-y|^{n+2s}}  dx dy+\int_{\mathbb{R}^n}\triangledown u_{\epsilon}\triangledown\phi dx,
\end{align*}
thanks to the fact that \(u\) is a weak solution of \eqref{equ:2.1}. This means that \(u_\varepsilon\) is a smooth solution of \eqref{eq:3.2}.

As a consequence of this, we may apply Corollary \ref{cor:4.3} and obtain that \( u_\varepsilon \in C^\alpha(B_{1/4}) \) for any \( 0 < \alpha < \min\{2s, 1\} \) and
$$
\|u_\varepsilon\|_{C^\alpha(B_{1/4})} \leqslant C\left(\|u_\varepsilon\|_{L^\infty(\mathbb{R}^n)} + \|f_\varepsilon\|_{L^\infty(B_1)}\right),
$$
for a suitable \( C > 0 \) depending on \( n \), \( s \) and \( \alpha \). Hence, employing the Arzel\`a--Ascoli Theorem, the desired result follows by sending \( \varepsilon \to 0 \).
\end{proof}

	\begin{theorem}[Hölder regularity for $p\in(0,1)$]\label{thm:1.1}
	Let $u \in \mathcal{X}^{1,2}(\mathbb{R}^n)$ be a weak solution of \eqref{equ1} with $p\in(0,1)$. Then $u \in C^\alpha(\mathbb{R}^n)$, for any $\alpha \in (0, \min\{2s, 1\})$, and $$\Vert u\Vert _{C^{\alpha}(\mathbb{R}^n)}\leqslant C(\|u\|_{L^{\infty}(\mathbb{R}^n)}+\lambda\|h\|_{L^{\infty}}\|u\|_{L^{\infty}(\mathbb{R}^n)}^p+\|u\|_{L^{\infty}(\mathbb{R}^n)}^{2^*-1}),$$ for some $C>0$, depending only on $n,s,\alpha$.
	\end{theorem}

\begin{proof}
The H\"older regularity of $u$ follows from Proposition \ref{pro:2.4} and the $L^{\infty}(\mathbb{R}^n)$-regularity in \cite[Theorem 2.2]{Su-Xu}, being $u$ a solution to \eqref{equ1} with $p\in(0,1).$ 
\end{proof}

		
\subsection{The regularity of viscosity solutions for $ p\in[1, 2^*-1)$}\label{sec4}

The goal of this section is to establish the $C^{2,\alpha}$-regularity result for solutions of problem \eqref{equ1} when $1\leqslant p < 2^*-1,$  as stated in Theorem \ref{thm:1.2}. Using techniques of smoothing and truncation \cite{MR4808805}, we can obtain the $C^{2,\alpha}$ regularity of the solution of \eqref{equ1}. However, here we adopt a different approach from \cite{MR4808805} by utilizing the properties of the Fourier transform and the Riesz kernel. Throughout this section $\mathcal{S}$ stands for the Schwartz space of rapidly decreasing $C^{\infty}$ functions in $\mathbb{R}^n$. We define the Riesz potentials by
\begin{equation*}
(I_\alpha f)(x) = \frac{1}{\gamma(\alpha)} \int_{\mathbb{R}^n} |x - y|^{-n+\alpha} f(y)  dy, \quad 0<\alpha<n
\end{equation*}
with
\[
\gamma(\alpha) = \pi^{n/2} 2^\alpha \frac{\Gamma(\alpha/2)}{\Gamma \left( \frac{n}{2} - \frac{\alpha}{2} \right)}.
\]

In Fourier space, we have $\widehat{\mathcal{L}u}(\xi) = (|\xi|^2 + |\xi|^{2s}) \hat{u}(\xi).$ Define
$
m(\xi) = |\xi|^2 + |\xi|^{2s},
$
and
$
M(\xi) = \frac{1}{m(\xi)}.
$
By Lemma \ref{proposition}, we have \begin{equation}\label{sou}u = \mathcal{F}^{-1}[M \hat{f}] = \mathcal{Z} * f,\end{equation} where $\mathcal{Z}: = \mathcal{F}^{-1}[M]$.

Take $\phi \in C_c^\infty(\mathbb{R}^n)$ with: $\phi(\xi) = 1$ for $|\xi| \leqslant 1$, $\phi(\xi) = 0$ for $|\xi| \geqslant 2$, $0 \leqslant \phi \leqslant 1$.
Let $\psi(\xi) = 1 - \phi(\xi)$. Define:
\[
M_1(\xi) = \frac{\phi(\xi)}{m(\xi)}, \quad M_2(\xi) = \frac{\psi(\xi)}{m(\xi)}.
\]
Clearly $M = M_1 + M_2$.

Define the kernels:
\[
\mathcal{Z}_1 = \mathcal{F}^{-1}[M_1], \quad \mathcal{Z}_2 = \mathcal{F}^{-1}[M_2].
\]
Then $ \mathcal{Z} = \mathcal{Z}_1 + \mathcal{Z}_2 $ and $u = \mathcal{Z}_1 * f + \mathcal{Z}_2 * f$.

For $|\xi| \leqslant 2$:
\[
M_1(\xi) = \frac{\phi(\xi)}{|\xi|^{2s}(1 + |\xi|^{2-2s})} = \frac{\phi(\xi)}{1 + |\xi|^{2-2s}} \cdot |\xi|^{-2s}:=A(\xi)\cdot|\xi|^{-2s}.
\]
Denote $z:=\mathcal{F}^{-1}[A(\xi)]$.
Thus,
\[
\mathcal{Z}_1 = z * R_{2s},
\]
where $R_{2s}(x) = c_{n,s} |x|^{-(n-2s)}$ is the kernel of the Riesz potential $I_{2s} = (-\Delta)^{-s}$.

On $\supp \psi$, $m(\xi) \geqslant |\xi|^2$, so $|M_2(\xi)| \leqslant |\psi(\xi)|/|\xi|^2$.  
In fact, $M_2 \in C^\infty$ and all its derivatives decay at least as $|\xi|^{-2}$ at infinity.

To obtain the $C^{2,\alpha}$-regularity of the solution of \eqref{equ1}, we require the following two evident conclusions of Schwartz space, for which we do not provide detailed proofs here.
\begin{lemma}\label{schwa} We have \(z=\mathcal{F}^{-1}[A(\xi)]\in C^\infty(\mathbb R^n)\). Moreover, for every multi-index \(\beta\), \(D^\beta z\in L^1(\mathbb R^n)\), with
\[
\|D^\beta z\|_{L^1}
\leqslant \|D^\beta\mathcal F^{-1}\varphi\|_{L^1}.
\]\end{lemma}
 \begin{proof} Let \(p_t\) denote the heat kernel associated with \((-\Delta)^{1-s}\). Then
\[
\widehat p_t(\xi)=e^{-t|\xi|^{2-2s}},
\qquad p_t\geqslant0,\qquad \int_{\mathbb R^n}p_t(x)\,dx=1.
\]Define
\[
G(x):=\int_0^\infty e^{-t}p_t(x)\,dt.
\]By Tonelli’s theorem, \(G\in L^1(\mathbb R^n)\) and
\[
\|G\|_{L^1}
=\int_0^\infty e^{-t}\int_{\mathbb R^n}p_t(x)\,dx\,dt=1.
\]Consequently, Fubini’s theorem applies and gives
\[
\widehat G(\xi)
=\int_0^\infty e^{-t}\widehat p_t(\xi)\,dt
=\frac1{1+|\xi|^{2-2s}}.
\]It follows that \(A=\varphi\widehat G\), and hence
\[
z=(\mathcal F^{-1}\varphi)*G.
\]Since \(\mathcal F^{-1}\varphi\in\mathcal S(\mathbb R^n)\), all its derivatives are bounded and integrable. Together with \(G\in L^1(\mathbb R^n)\), this justifies differentiation under the convolution integral and yields
\[
z\in C^\infty(\mathbb R^n),
\qquad
D^\beta z=(D^\beta\mathcal F^{-1}\varphi)*G.
\]Applying the Young’s inequality, we obtain
\[
\|D^\beta z\|_{L^1}
\leqslant \|D^\beta\mathcal F^{-1}\varphi\|_{L^1}\|G\|_{L^1}
=\|D^\beta\mathcal F^{-1}\varphi\|_{L^1}<\infty.
\]This proves the lemma. \end{proof}

	
	\begin{lemma}\label{con}
		For any $f\in L^{\infty}(\mathbb{R}^n)$ and any $\alpha\in(0,1),$ there exists constant $C=C(n,\alpha,s)>0,$ such that $$\|\mathcal{Z}_2*f\|_{C^{1,\alpha}(\mathbb{R}^n)}\leqslant C\|f\|_{L^{\infty}(\mathbb{R}^n)}.$$	
		\end{lemma}
\begin{proof}
Choose \(\varphi\in C_c^\infty(\mathbb R^n)\) such that
\[
\operatorname{supp}\varphi\subset\{\xi:1/2\le|\xi|\le2\},
\qquad
\sum_{j\in\mathbb Z}\varphi(2^{-j}\xi)=1
\quad\text{for }\xi\ne0.
\]Since \(M_2\) vanishes on \(\{|\xi|\le1\}\), we have
$
M_2(\xi)=\sum_{j=0}^\infty M_2(\xi)\phi(2^{-j}\xi).
$
Set
\[
K_j:=\mathcal F^{-1}\!\left[M_2(\xi)\varphi(2^{-j}\xi)\right],
\qquad
u_j:=K_j*f.
\]We first prove
\begin{equation}\label{one}
\|D^\beta K_j\|_{L^1}
\leqslant C_\beta2^{(|\beta|-2)j},
\qquad |\beta|\le2,\quad j\ge0.
\end{equation}
Define
\[
a_j(\xi):=2^{2j}M_2(2^j\xi)\varphi(\xi)
=
\frac{\psi(2^j\xi)\varphi(\xi)}
{|\xi|^2+2^{-j(2-2s)}|\xi|^{2s}},
\]with \(a_j=0\leqslant \)the origin. Then \(a_j\in C_c^\infty(\mathbb R^n)\), and all \(a_j\) are supported in the fixed annulus
\[
E:=\{\xi:1/2\leqslant|\xi|\leqslant2\}.
\]For \(j\geqslant2\), we have \(\psi(2^j\xi)=1\) on \(E\) and 
\[
a_j(\xi)=
\frac{\phi(\xi)}
{|\xi|^2+2^{-j(2-2s)}|\xi|^{2s}}.
\]Note that all derivatives with respect to \(\xi\) of
$
(\xi,t)\longmapsto\frac1{|\xi|^2+t|\xi|^{2s}}
$
are continuous on \(E\times[0,1]\), and $|\xi|^2+t|\xi|^{2s}\geqslant\frac{1}{4}$. Since \(2^{-j(2-2s)}\in(0,1]\), imply
\[
\sup_{j\geqslant2}\|D^\gamma a_j\|_{L^\infty}\leqslant C_\gamma
\]for every multi-index \(\gamma\). For \(a_0,a_1\), we obtain
\begin{equation}\label{gamma}
\sup_{j\geqslant0}\|D^\gamma a_j\|_{L^\infty}\leqslant C_\gamma.
\end{equation}Choose an integer \(N>n/2\). Integration by parts in the inverse Fourier transform gives
\[
(1+|x|^2)^N D^\beta\mathcal F^{-1}a_j(x)
=
\mathcal F^{-1}\!\left[
(1-\Delta_\xi)^N\bigl((i\xi)^\beta a_j\bigr)
\right](x).
\]By \eqref{gamma}, we have,
\[
\left\|
(1-\Delta_\xi)^N\bigl((i\xi)^\beta a_j\bigr)
\right\|_{L^1}
\leqslant C_{\beta,N}.
\]Consequently,
\[
|D^\beta\mathcal F^{-1}a_j(x)|
\leqslant C_{\beta,N}(1+|x|^2)^{-N},
\]and therefore
\[
\|D^\beta\mathcal F^{-1}a_j\|_{L^1}\leqslant C_\beta.
\]The scaling identity
\[
K_j(x)=2^{(n-2)j}(\mathcal F^{-1}a_j)(2^jx)
\] yields
\[
\|D^\beta K_j\|_{L^1}
=
2^{(|\beta|-2)j}
\|D^\beta\mathcal F^{-1}a_j\|_{L^1}
\leqslant C_\beta2^{(|\beta|-2)j},
\]which proves \eqref{one}.
By
\[
\mathcal Z_2*f=\sum_{j=0}^\infty u_j,
\] and the Young’s inequality, we have
\begin{equation}\label{ghuj}
\|u_j\|_{L^\infty}
\leqslant C2^{-2j}\|f\|_{L^\infty},\,
\|\nabla u_j\|_{L^\infty}
\leqslant C2^{-j}\|f\|_{L^\infty} \text{ and } \|D^2u_j\|_{L^\infty}
\leqslant C\|f\|_{L^\infty}.
\end{equation}
In particular, 
\[
\nabla(\mathcal Z_2*f)=\sum_{j=0}^\infty\nabla u_j.
\]For any \(x,y\in\mathbb R^n\), the triangle inequality and the mean value estimate, together with \eqref{ghuj}, imply
\[
|\nabla u_j(x)-\nabla u_j(y)|
\leqslant C\|f\|_{L^\infty}\min\{2^{-j},|x-y|\}.
\]Using
\[
\min\{a,b\}\leqslant a^{1-\alpha}b^\alpha,
\qquad a,b>0,
\]we obtain
\[
[\nabla u_j]_{C^\alpha}
\leqslant C2^{-(1-\alpha)j}\|f\|_{L^\infty}.
\]Combining this estimate with \eqref{ghuj} and  \(j\geqslant0\), we conclude that
\[
\begin{aligned}
\|\mathcal Z_2*f\|_{C^{1,\alpha}}
&\leqslant C\|f\|_{L^\infty}
\sum_{j=0}^\infty
\left(2^{-2j}+2^{-j}+2^{-(1-\alpha)j}\right)\\
&\leqslant C(n,s,\alpha)\|f\|_{L^\infty}.
\end{aligned}
\]This completes the proof
\end{proof}

	Employing the above two lemmas, we get the following regularity result of weak solutions of \eqref{equ1}:
\begin{proposition}\label{pro 4.1}
Assume that \((h_1)\) holds and $s\in(0,1).$ If $u\in \mathcal{X}^{1,2}(\mathbb{R}^n)\cap L^{\infty}(\mathbb{R}^n)$ and $\mathcal{L}u=f:=\lambda h u^p+u^q,$ where $q:=2^*-1.$  Then \(u\in C^{1,\alpha}(\mathbb R^n)\) for every \(\alpha\in(0,1)\), and
\[
\begin{aligned}
\|u\|_{C^{1,\alpha}(\mathbb R^n)}
\leqslant C\Bigl[
&\lambda\bigl(\|h\|_{L^1(\mathbb R^n)}
+\|h\|_{L^\infty(\mathbb R^n)}\bigr)
\|u\|_{L^\infty(\mathbb R^n)}^p+\|u\|_{L^{2^*}(\mathbb R^n)}^{2^*-1}
+\|u\|_{L^\infty(\mathbb R^n)}^{2^*-1}
\Bigr],
\end{aligned}
\]
where \(C>0\) depends only on \(n,s,\alpha\).
\end{proposition}

\begin{proof}
Set
\[
t:=\frac{2n}{n+2}.
\]
By the Sobolev embedding, we have \(u\in L^{2^*}(\mathbb R^n)\). Since \(qt=2^*\), we have
\[
\|u^q\|_{L^t}=\|u\|_{L^{2^*}}^q.
\]
Moreover, assumption \((h_1)\) gives \(h\in L^1(\mathbb R^n)\cap L^\infty(\mathbb R^n)\), and hence
\[
\|h\|_{L^t}
\leqslant\|h\|_{L^1}^{1/t}\|h\|_{L^\infty}^{1-1/t}
\leqslant \|h\|_{L^1}+\|h\|_{L^\infty}.
\]

Consequently,\begin{equation}\label{pp}
\begin{aligned}
\|f\|_{L^t}
&\leqslant
\lambda\bigl(\|h\|_{L^1}+\|h\|_{L^\infty}\bigr)
\|u\|_\infty^p
+\|u\|_{L^{2^*}}^q,\\
\|f\|_\infty
&\leqslant \lambda\|h\|_\infty\|u\|_\infty^p+\|u\|_\infty^q.
\end{aligned}
\end{equation}
Thus \(f\in L^t(\mathbb R^n)\cap L^\infty(\mathbb R^n)\). By Lemma \ref{proposition},
we have
\[
\mathcal{Z}=\mathcal{Z}_1+\mathcal{Z}_2,\qquad \mathcal{Z}_1=z*R_{2s}.
\]
We first estimate the Riesz potential
\[
w:=R_{2s}*f.
\]
Since \(2s<n\),
\[
R_{2s}\in L^1(B_1).
\]
Furthermore, with \(t'=2n/(n-2)\), the inequality
\[
(n-2s)t'>n
\]
follows from \(n+2>4s\). Therefore,
\[
R_{2s}\in L^{t'}(\mathbb R^n\setminus B_1).
\]
Splitting the integral into these two regions and applying the Hölder’s inequality, we obtain, for every \(x\in\mathbb R^n\),
\[
\begin{aligned}
\int_{\mathbb R^n}R_{2s}(x-y)|f(y)|\,dy
&\leqslant
\|R_{2s}\|_{L^1(B_1)}\|f\|_\infty+
\|R_{2s}\|_{L^{t'}(\mathbb R^n\setminus B_1)}
\|f\|_{L^t}\\
&\leqslant C\bigl(\|f\|_\infty+\|f\|_{L^t}\bigr).
\end{aligned}
\]
Therefore
\begin{equation}\label{dm}
\|w\|_\infty
\leqslant C\bigl(\|f\|_\infty+\|f\|_{L^t}\bigr).
\end{equation}
By Lemma \ref{schwa}, \(z\in L^1(\mathbb R^n)\). The by the Fubini’s theorem and gives
\[
\mathcal{Z}_1*f=(z*R_{2s})*f=z*(R_{2s}*f)=z*w.
\]
where,
\[
\int_{\mathbb R^n}|z(y)|
\int_{\mathbb R^n}R_{2s}(x-y-\zeta)|f(\zeta)|\,d\zeta\,dy
\leqslant C\|z\|_1\bigl(\|f\|_\infty+\|f\|_{L^t}\bigr)<\infty.
\]
Set \(v:=Z_1*f=z*w\). By Lemma \ref{schwa}, \(D^\gamma z\in L^1(\mathbb R^n)\) for every multi-index \(\gamma\). And applying the Young’s inequality, we obtain
\[
D^\gamma v=(D^\gamma z)*w,
\qquad
\|D^\gamma v\|_\infty
\leqslant \|D^\gamma z\|_1\|w\|_\infty,
\qquad |\gamma|\leqslant2.
\]
Consequently,
\begin{equation}\label{dh}
\|v\|_\infty+\|\nabla v\|_\infty+\|D^2v\|_\infty
\leqslant C\|w\|_\infty.
\end{equation}
For \(0<|x-y|\leqslant1\), the mean value theorem yields
\[
|\nabla v(x)-\nabla v(y)|
\leqslant C\|D^2v\|_\infty|x-y|
\leqslant C\|w\|_\infty|x-y|^\alpha.
\]
For \(|x-y|>1\), we instead use
\[
|\nabla v(x)-\nabla v(y)|
\le2\|\nabla v\|_\infty
\leqslant C\|w\|_\infty|x-y|^\alpha.
\]
Combining these estimates with \eqref{dm},\eqref{dh}, we conclude that
\[
\|\mathcal{Z}_1*f\|_{C^{1,\alpha}(\mathbb R^n)}
\leqslant C\bigl(\|f\|_\infty+\|f\|_{L^t}\bigr).
\]
On the other hand, Lemma \ref{con} gives
\[
\|\mathcal{Z}_2*f\|_{C^{1,\alpha}(\mathbb R^n)}
\leqslant C\|f\|_\infty.
\]
Since \(u=\mathcal{Z}_1*f+\mathcal{Z}_2*f\), it follows that
\[
\|u\|_{C^{1,\alpha}(\mathbb R^n)}
\leqslant C\bigl(\|f\|_{L^t}+\|f\|_\infty\bigr).
\]
Finally, inserting the bounds in \eqref{pp} and adjusting \(C\) gives the asserted estimate. 
\end{proof}

\begin{proposition}
\label{pro 4.2}
Assume that $(h_1)$ holds and $0 < s < 1$.  
If $u \in  \mathcal{X}^{1,2}(\mathbb{R}^n)\cap L^\infty(\mathbb{R}^n)$ and $\mathcal{L}u = f:=\lambda hu^p+u^{2^*-1}\in C^{0,\alpha}(\mathbb R^n)$ with $0 < \alpha<1$, then $u\in C^{2,\alpha}(\mathbb{R}^n)$ and 
\[
\begin{aligned}
\|u\|_{C^{2,\alpha}(\mathbb R^n)}
\leqslant C\Bigl[\|f\|_{C^{0,\alpha}(\mathbb{R}^n)}+\lambda\bigl(\|h\|_{L^1(\mathbb R^n)}
+\|h\|_{L^\infty(\mathbb R^n)}\bigr)
\|u\|_{L^\infty(\mathbb R^n)}^p+\|u\|_{L^{2^*}(\mathbb R^n)}^{2^*-1}
\Bigr],
\end{aligned}
\]
where \(C>0\) depends only on \(n,s,\alpha\).

\end{proposition}

\begin{proof}
As shown in the proof of Proposition \ref{pro 4.1}, we have 
\begin{equation}\label{momo}
\|f\|_{L^t}\leqslant\lambda(\|h\|_{L^1}+\|h\|_{L^{\infty}})\|u\|_{L^{\infty}}^p+\|u\|_{L^{2^*}}^q,
\end{equation}
and 
\begin{equation}
\|w\|_{L^{\infty}}\leqslant C(\|f\|_{L^t}+\|f\|_{L^{\infty}}).
\end{equation}
Let $v:=\mathcal{Z}_1*f=z*w,$ by Lemma \ref{schwa}, $D^{\gamma}z\in L^1(\mathbb{R}^n)$ for every multi-indec $\gamma.$ Therefore, the Young's inequality gives
$$D^{\gamma}v=(D^{\gamma}z)*w,\,\, \Sigma_{|\gamma|\leqslant3}\|D^{\gamma}v\|_{L^{\infty}}\leqslant C\|w\|_{L^{\infty}}.$$
For $|x-y|\leqslant1,$ the mean value theorem yields
$$|D^2v(x)-D^2v(y)|\leqslant C\|D^3 v\|_{L^{\infty}}|x-y|\leqslant C\|w\|_{L^{\infty}}|x-y|^{\alpha}.$$
For $|x-y|>1,$ the same estimate follows from the boundedness of $D^2v.$ Consequently,
\begin{equation}\label{popo}
\|\mathcal{Z}_1*f\|_{C^{2,\alpha}}\leqslant C\|w\|_{L^{\infty}}\leqslant C(\|f\|_{L^t}+\|f\|_{L^{\infty}}).
\end{equation}
We use the kernels $K_j$ from the proof of Lemma \ref{con} and write $\mathcal{Z}_2*f=\Sigma_{j=0}^{\infty}w_j,\,\, w_j=K_j*f.$
The same uniform symbol estimates and integration by parts argument give 
$$|D^{\gamma}K_j(x)|\leqslant C_{\gamma,N}2^{(n+|\gamma|-2)j}(1+2^j|x|)^{-N},\,|\gamma|\leqslant3.$$
Choosing $N>n+\alpha,$ we obtain \begin{equation}\label{equs}
\int_{\mathbb R^n}|D^\gamma K_j(y)|\,|y|^\alpha\,dy
\leqslant C_\gamma2^{(|\gamma|-2-\alpha)j}.
\end{equation}
Since the Fourier transform of \(K_j\) vanishes near the origin,
\[
\int_{\mathbb R^n}D^\gamma K_j(y)\,dy=0.
\]
Thus
\[
D^\gamma w_j(x)
=
\int_{\mathbb R^n}D^\gamma K_j(y)
\bigl(f(x-y)-f(x)\bigr)\,dy,
\]
and \eqref{equs} implies
\begin{equation}\label{equd}
\|D^\gamma w_j\|_\infty
\leqslant C[f]_\alpha2^{(|\gamma|-2-\alpha)j},
\qquad |\gamma|\leqslant3.
\end{equation}
In particular, the series and its derivatives up to order two converge uniformly, and
\[
\sum_{|\gamma|\le2}
\|D^\gamma(\mathcal{Z}_2*f)\|_\infty
\leqslant C[f]_\alpha.
\]
Let \(d=|x-y|\in(0,1)\), and choose \(J\geqslant0\) such that
\[
2^{-J-1}<d\leqslant2^{-J}.
\]
By \eqref{equd},
\[
\begin{aligned}
|D^2(\mathcal{Z}_2*f)(x)-D^2(\mathcal{Z}_2*f)(y)|&\leqslant
d\sum_{j=0}^{J}\|D^3w_j\|_\infty
+2\sum_{j>J}\|D^2w_j\|_\infty\\
&\leqslant
C[f]_\alpha
\left(
d\sum_{j=0}^{J}2^{(1-\alpha)j}
+\sum_{j>J}2^{-\alpha j}
\right)
\leqslant C[f]_\alpha d^\alpha.
\end{aligned}
\]
For \(d\geqslant1\), we use the boundedness of the second derivatives. Hence
\begin{equation}\label{koko}
\|\mathcal{Z}_2*f\|_{C^{2,\alpha}}
\leqslant C\|f\|_{C^{0,\alpha}}.
\end{equation}
Combining \eqref{popo} and \eqref{koko}, we obtain
\[
\|u\|_{C^{2,\alpha}}
\leqslant C\bigl(\|f\|_{L^t}+\|f\|_{C^{0,\alpha}}\bigr).
\]
Substituting \eqref{momo} proves the asserted estimate.
\end{proof}

\begin{lemma}[$C^{1,\theta}$- regularity]\label{regu}Assume that \((h_1)\) holds. Let $u \in \mathcal{X}^{1,2}(\mathbb{R}^n)$ be a nontrivial solution of \eqref{equ1} with $p\in[1,2^*-1)$. Then $u \in C^{1,\theta}(\mathbb{R}^n)$, for any $\theta \in (0, 1)$ and \[
\begin{aligned}
\|u\|_{C^{1,\theta}(\mathbb R^n)}
\le C\Bigl[
&\lambda\bigl(\|h\|_{L^1(\mathbb R^n)}
+\|h\|_{L^\infty(\mathbb R^n)}\bigr)
\|u\|_{L^\infty(\mathbb R^n)}^p+\|u\|_{L^{2^*}(\mathbb R^n)}^{2^*-1}
+\|u\|_{L^\infty(\mathbb R^n)}^{2^*-1}
\Bigr],
\end{aligned}
\]
where \(C>0\) depends only on \(n,s,\theta\).	\end{lemma}
\begin{proof}
Set
\[
t:=\frac{2n}{n+2},\,\, q=2^*-1.
\]
By the uniform $L^{\infty}$-estimate of $u$ (\cite[Theorem 2.2]{Su-Xu}) and the assumption ($h_1$) that  $h\in L^{\infty}(\mathbb{R}^n)$, we get $$\|\lambda h u^p+u^{2^*-1}\|_{L^{\infty}}\leqslant C\lambda \|h\|_{L^{\infty}}\|u\|_{L^{\infty}}^p+\|u\|_{L^{\infty}}^{2^*-1}.$$ Furthermore, the Sobolev embedding implies \(u\in L^{2^*}(\mathbb R^n)\). Since \(qt=2^*\) and \(h\in L^1(\mathbb R^n)\cap L^\infty(\mathbb R^n)\), we obtain
\[
\begin{aligned}
\|f\|_{L^t}
&\le\lambda\|h\|_{L^t}\|u\|_\infty^p+\|u\|_{L^{2^*}}^q\\
&\le
\lambda\bigl(\|h\|_{L^1}+\|h\|_\infty\bigr)\|u\|_\infty^p
+\|u\|_{L^{2^*}}^q,
\end{aligned}
\]
where $f:=\lambda h u^p+u^{2^*-1}.$
Thus \(f\in L^t(\mathbb R^n)\cap L^\infty(\mathbb R^n)\), and the Riesz potential estimates established in the proof of Proposition \ref{pro 4.1} apply. For every \(\theta\in(0,1)\), they yield
\[
\|u\|_{C^{1,\theta}(\mathbb R^n)}
\le C(n,s,\theta)
\bigl(\|f\|_{L^t(\mathbb R^n)}+\|f\|_{L^\infty(\mathbb R^n)}\bigr).
\]
Inserting the preceding bounds and adjusting the constant gives the asserted estimate.
\end{proof}

Next, we will give the proof of Theorem \ref{thm:1.2}. 
For this, we combine a suitable truncation argument for the solution $u$ with the $C^{1,\alpha}$-regularity argument. By analogous reasoning to the proof of \cite[Theorem 1.4]{SerenaDipierro2025}, the $C^{2,\alpha}$- regularity of the solution follows. For the convenience of the reader, we sketch the proof in the following subsections.

To begin with, we introduce some notations of semi-norms as follows. 
For $\alpha \in (0, 1)$, $k \in \mathbb{N}$, $x_0 \in B_{3/4}$, and $R \in (0, \frac{1}{20})$, we denote the interior norms as follows:
\[
[u]_{\alpha;B_R(x_0)}
:=
\sup_{\substack{x,y\in B_R(x_0)\\x\ne y}}
\frac{|u(x)-u(y)|}{|x-y|^\alpha},
\]
\[
|u|'_{k;B_R(x_0)} := \sum_{j=0}^{k} R^j \|D^ju\|_{L^\infty(B_R(x_0))}
\]
and
\[
|u|'_{k,\alpha;B_R(x_0)} := |u|'_{k;B_R(x_0)} + R^{k+\alpha}[D^ku]_{\alpha;B_R(x_0)}.
\]

\textbf{Step 1. A mollifier technique and a truncation argument.}
Let $u \in \mathcal{X}^{1,2}(\mathbb{R}^n)$ solve \eqref{equ1} and $\eta_\varepsilon$ be a standard mollifier. For every $x \in \mathbb{R}^n$, $R \in (0, \frac{1}{20})$, and $\varepsilon \in (0, R)$, we denote
\[
u_\varepsilon(x) := (\eta_\varepsilon * u)(x) = \int_{|y| \leqslant \varepsilon} \eta_\varepsilon(y)u(x-y)  dy.
\]
Also, we set
\[
g_\varepsilon := \eta_\varepsilon * (\lambda hu^p + u^{2^*-1}).
\]
Then, we have that
\[
-\Delta u_\varepsilon + (-\Delta)^s u_\varepsilon = g_\varepsilon \quad \text{in } B_1.
\]

Moreover, the following regularity estimates on $u_\varepsilon$ and $g_\varepsilon$ follow as a direct consequence of their definitions.
\begin{lemma}\label{ghgh}
Let $u \in L^\infty(\mathbb{R}^n)$. Then, $u_\varepsilon \in L^\infty(\mathbb{R}^n)$ and
\[
\|u_\varepsilon\|_{L^\infty(\mathbb{R}^n)} \leqslant \|u\|_{L^\infty(\mathbb{R}^n)}.
\]
If in addition $u \in C^1(\mathbb{R}^n)$, then for every $y \in \mathbb{R}^n$,
\[
\|g_\varepsilon\|_{C^\alpha(B_1(y))} \leqslant c_{p,2^*} \left(\lambda\|h\|_{C^1(\R^n)} \|u\|^p_{C^1(\mathbb{R}^n)} + \|u\|_{C^1(\mathbb{R}^n)}^{2^*-1} \right)\]
where $c_{p,2^*}$ is some positive constant depending only on $p, 2^*$.
\end{lemma}
\begin{proof}
A direct computation gives
$$|u_{\epsilon}(x)|\leqslant\int_{\mathbb{R}^n}|\eta_{\epsilon}(y)u(x-y)|dy\leqslant\|u\|_{L^{\infty}(\mathbb{R}^n)}\int_{\mathbb{R}^n}|\eta_{\epsilon}(y)|dy=\|u\|_{L^{\infty}(\mathbb{R}^n)}.$$
The properties of convolution yield 
\begin{equation*}
\begin{aligned}
\|g_{\epsilon}\|_{C^{\alpha}(B_1(y))}&\leqslant\lambda\|h\|_{C^1}(2\|u\|^p_{L^{\infty}}+\|Du^p\|_{L^{\infty}})+2^*\|u\|_{C^1}^{2^*-1}\\
&\leqslant C\lambda\|h\|_{C^1}(1+p)\|u\|_{C^1}^p+2^*\|u\|_{C^1}^{2^*-1}\\
&:=c_{p,2^*}(\lambda\|h\|_{C^1}\|u\|_{C^1}^p+\|u\|_{C^1}^{2^*-1}).
\end{aligned}
\end{equation*}
\end{proof}

We now use a cut off argument for $u_\varepsilon$ to get the $C^{2,\alpha}$-estimate for $u_\varepsilon$. Consider a cut-off function $\phi \in C_0^\infty (\mathbb{R}^n)$ satisfying
\[
\phi \equiv 1 \text{ in } B_{3/2}, \quad \text{supp}\phi \subset B_2, \quad \text{and} \quad 0 \leqslant \phi \leqslant 1 \text{ in } \mathbb{R}^n
\]
and let
\[
\phi^R(x) := \phi \left( \frac{x - x_0}{R} \right).
\]

We point out that

\[
B_{4R}(x_0) \subset B_1 \quad \text{and} \quad \text{supp}(\phi^R) \subset B_{2R}(x_0).
\]
With this notation, one obtains the following result.

\begin{lemma}[{\cite[Lemma 5.1]{MR4808805}}]
 \label{dghs}Let $\alpha \in (0, 1)$ and $g \in C_{\text{loc}}^{\alpha}(B_1)$. Let $u \in C^{2,\alpha}(\overline{B}_1) \cap L^\infty(\mathbb{R}^n)$ be a solution of
\[
-\Delta u + (-\Delta)^s u = g \quad \text{in } B_1.
\]
Then, there exists $\psi \in C^{\alpha}(B_R(x_0))$ such that $v := \phi^R u$ satisfies
\[
-\Delta v + (-\Delta)^s v = \psi \quad \text{in } B_1.
\]
In particular,
\begin{equation}\label{equoo}
R^2 |\psi|'_{0,\alpha; B_R(x_0)} \leqslant C_{n,s} \left( R^2 |g|'_{0,\alpha; B_R(x_0)} + \|u\|_{L^\infty(\mathbb{R}^n)} \right), 
\end{equation}
for some positive constant $C_{n,s}$.
\end{lemma}

As a consequence of Lemma \ref{dghs}, setting $v_\varepsilon := \phi^R u_\varepsilon$, we have that there exists $\psi_\varepsilon \in C^{\alpha}(B_R(x_0))$ such that $v_\varepsilon$ satisfies
\[
-\Delta v_\varepsilon + (-\Delta)^s v_\varepsilon = \psi_\varepsilon \quad \text{in } B_1.
\]
In particular, employing \eqref{equoo}, one finds that
\begin{equation}\label{equhh}
R^2 |\psi'_\varepsilon|_{0,\alpha; B_R(x_0)} \leqslant C_{n,s} \left( R^2 |g_\varepsilon|'_{0,\alpha; B_R(x_0)} + \|u_\varepsilon\|_{L^\infty(\mathbb{R}^n)} \right). 
\end{equation}

We also observe that, since $v_\varepsilon \in C_0^\infty(B_{2R}(x_0))$, for all $\delta > 0$, there exists $C_{\delta} > 0$ such that
\begin{align*}
R^2 |(-\Delta)^s v_\varepsilon|'_{0,\alpha; B_R(x_0)} &= R^2 \|(-\Delta)^s v_\varepsilon\|_{L^\infty(B_R(x_0))} + R^{2+\alpha}[(-\Delta)^s v_\varepsilon]_{\alpha; B_R(x_0)}\\
& \leqslant C_{n,s} |u_\varepsilon|'_{2,\alpha_0; B_{2R}(x_0)} \leqslant \delta |u_\varepsilon|'_{2,\alpha; B_{2R}(x_0)} + C_{\delta} \|u_\varepsilon\|_{L^\infty(B_{2R}(x_0))},
\end{align*}
where
\[
\alpha_0 := 
\begin{cases} 
0 & \alpha < 2 - 2s, \\
\alpha - (1-s) & \alpha \geqslant 2 - 2s.
\end{cases}
\]

Thus, combining this with \cite[Theorem 4.6]{MR1814364}, one could deduce that for all $\delta > 0$ there exists $C_{\delta} > 0$ such that
\begin{align*}
|v_\varepsilon|'_{2,\alpha; B_{R/2}(x_0)} &\leqslant C \left( \|v_\varepsilon\|_{L^\infty(B_R(x_0))} + R^2 \left( |\psi_\varepsilon|'_{0,\alpha; B_R(x_0)} + |(-\Delta)^s v_\varepsilon|'_{0,\alpha; B_R(x_0)} \right) \right) \\
&\leqslant C \left( \|v_\varepsilon\|_{L^\infty(B_R(x_0))} + R^2 |\psi_\varepsilon|'_{0,\alpha; B_R(x_0)} + \delta |u_\varepsilon|'_{2,\alpha; B_{2R}(x_0)} + C_{\delta} \|u_\varepsilon\|_{L^\infty(B_{2R}(x_0))} \right).
\end{align*}

Therefore, for every $x_0 \in B_{3/4}$, recalling the definition of $v_\epsilon$ and exploiting \eqref{equhh} and Lemma~\ref{ghgh}, we conclude that, for all $\delta > 0$ above, there exists $C_\delta$ such that
\begin{equation}\label{equggh}
\begin{aligned}
&|u_\epsilon|'_{2,\alpha;B_{R/2}(x_0)} = |v_\epsilon|'_{2,\alpha;B_{R/2}(x_0)} \\
&\leqslant C \left( R^2 |\psi_\epsilon|'_{0,\alpha;B_R(x_0)} + \delta |u_\epsilon|'_{2,\alpha;B_{2R}(x_0)} + C_\delta \|u\|_{L^\infty(\mathbb{R}^n)} \right) \\
&\leqslant C \left( R^2 |g_\epsilon|'_{0,\alpha;B_R(x_0)} + \|u_\epsilon\|_{L^\infty(\mathbb{R}^n)} + \delta |u_\epsilon|'_{2,\alpha;B_{2R}(x_0)} + C_\delta \|u\|_{L^\infty(\mathbb{R}^n)} \right) \\
&\leqslant C \left( \lambda\|h\|_{C^1(\R^n)} \|u\|^p_{C^1(\mathbb{R}^n)} + \|u\|_{C^1(\mathbb{R}^n)}^{2^*-1} + C_\delta\|u\|_{L^\infty(\mathbb{R}^n)} + \delta |u_\epsilon|'_{2,\alpha;B_{2R}(x_0)} \right),
\end{aligned}
\end{equation}
for some $C > 0$ depending on $n$, $s$, $p$, $2^*$ and $\alpha$.

\textbf{Step 2.  Interior $C^{2,\alpha}$-regularity.}
The estimate in \eqref{equggh}, coupled with the following statement, will allow us to obtain that the $C^{2,\alpha}$-norm of $u_\epsilon$ is bounded uniformly in some ball with respect to $\epsilon$, and thus use Arzelà-Ascoli theorem to complete the proof of Theorem \ref{thm:1.2}. The technical details go as follows.
\begin{proposition}[{\cite[Proposition 5.2]{MR4808805}}]
 \label{prohh}Let $y \in \mathbb{R}^n$, $d > 0$, and $u \in C^{2,\alpha}(B_d(y))$. Suppose that, for any $\delta > 0$, there exists $\Lambda_\delta > 0$ such that, for any $x \in B_d(y)$ and any $r \in (0,d - |x-y|]$, we have that
\begin{equation}\label{hhhj}
|u|'_{2,\alpha;B_{r/8}(x)} \leqslant \Lambda_\delta + \delta |u|'_{2,\alpha;B_{r/2}(x)}. 
\end{equation}

Then, there exist constants $\delta_0, C > 0$, depending only on $n$, $\alpha$ and $d$, such that
\[
\|u\|_{C^{2,\alpha}(B_{d/8}(y))} \leqslant C \Lambda_{\delta_0}.\]
\end{proposition}
\begin{proof}[\textbf{Proof of Theorem \ref{thm:1.2}}]
We will apply Proposition \ref{prohh} with \(d:=1/10\). For every \(y\in B_{1/2}\), we have
\[
B_d(y)\subset B_{3/4},
\qquad
\frac d4<\frac1{20}.
\]
Since the equation holds in the whole space, the mollification parameter can be chosen independently of the truncation radius. Thus, we fix \(0<\varepsilon<1/4\), for every \(R\in(0,1/20)\),  \eqref{equggh}  gives
\[
\begin{aligned}
|u_\varepsilon|_{2,\alpha;B_{R/2}(x_0)}'
\leqslant{}&
\delta|u_\varepsilon|_{2,\alpha;B_{2R}(x_0)}'+C_\delta\Bigl(
\lambda\|h\|_{C^1(\mathbb R^n)}
\|u\|_{C^1(\mathbb R^n)}^p
+\|u\|_{C^1(\mathbb R^n)}^{2^*-1}
+\|u\|_{L^\infty(\mathbb R^n)}
\Bigr)
\end{aligned}
\]
for every \(x_0\in B_{3/4}\), where \(C_\delta\) is independent of \(\varepsilon,R,x_0\).

For \(x\in B_d(y)\) and \(0<r\leqslant d-|x-y|\), taking \(R=r/4\) shows that \eqref{hhhj} holds with \(u\) replaced by \(u_\varepsilon\) and
\[
\Lambda_\delta:=
C_\delta\Bigl(
\lambda\|h\|_{C^1(\mathbb R^n)}
\|u\|_{C^1(\mathbb R^n)}^p
+\|u\|_{C^1(\mathbb R^n)}^{2^*-1}
+\|u\|_{L^\infty(\mathbb R^n)}
\Bigr).
\]
Consequently, Proposition \ref{prohh} yields
\begin{equation}\label{lili}
\begin{aligned}
\|u_\varepsilon\|_{C^{2,\alpha}(B_{1/80}(y))}
\leqslant C\Bigl(
&\lambda\|h\|_{C^1(\mathbb R^n)}
\|u\|_{C^1(\mathbb R^n)}^p+\|u\|_{C^1(\mathbb R^n)}^{2^*-1}
+\|u\|_{L^\infty(\mathbb R^n)}
\Bigr)
\end{aligned}
\end{equation}
for every \(y\in B_{1/2}\). Here \(C_{\delta}\) has been absorbed into \(C\), which is independent of \(\varepsilon\) and \(y\).

By Lemma \ref{regu}, we have
\[
\|u_\varepsilon-u\|_{L^\infty(\mathbb R^n)}
\leqslant
\varepsilon\|\nabla u\|_{L^\infty(\mathbb R^n)}
\longrightarrow0.
\]
Estimate \eqref{lili} and the Arzelà–Ascoli theorem therefore imply that
\[
u\in C^{2,\alpha}(\overline{B_{1/80}(y)}).
\]
 Passing to the limit in the uniform H\"older estimate. Thus
\[
\begin{aligned}
\|u\|_{C^{2,\alpha}(B_{1/80}(y))}
\leqslant C\Bigl(
&\lambda\|h\|_{C^1(\mathbb R^n)}
\|u\|_{C^1(\mathbb R^n)}^p+\|u\|_{C^1(\mathbb R^n)}^{2^*-1}
+\|u\|_{L^\infty(\mathbb R^n)}
\Bigr).
\end{aligned}
\]
Since the constant is uniform in \(y\in B_{1/2}\), a covering argument gives
\[
\begin{aligned}
\|u\|_{C^{2,\alpha}(B_{1/2})}
\leqslant C\Bigl(
&\lambda\|h\|_{C^1(\mathbb R^n)}
\|u\|_{C^1(\mathbb R^n)}^p+\|u\|_{C^1(\mathbb R^n)}^{2^*-1}
+\|u\|_{L^\infty(\mathbb R^n)}
\Bigr).
\end{aligned}
\]
The ball $B_y$ is centered at the origin, but we may arbitrarily move it around $\mathbb{R}^n$.
Covering $\mathbb{R}^n$ with these balls, we obtain the desired result
\[
\begin{aligned}
\|u\|_{C^{2,\alpha}(\mathbb R^n)}
\leqslant C\Bigl(
&\lambda\|h\|_{C^1(\mathbb R^n)}
\|u\|_{C^1(\mathbb R^n)}^p+\|u\|_{C^1(\mathbb R^n)}^{2^*-1}
+\|u\|_{L^\infty(\mathbb R^n)}
\Bigr),
\end{aligned}
\]
where \(C>0\) depends only on \(n,s,\alpha,2^*,p\). 

Finally, the Sobolev embedding gives \(u\in L^{2^*}(\mathbb R^n)\). If \(u\) did not tend to zero at infinity, there would exist \(a>0\) and points \(x_j\), with \(|x_j|\to\infty\), such that \(u(x_j)\geqslant a\). Uniform continuity would provide a fixed radius \(r>0\) such that
\[
u\geqslant a/2\qquad\text{in }B_r(x_j).
\]
Choosing a subsequence of pairwise disjoint balls would contradict \(u\in L^{2^*}(\mathbb R^n)\). Therefore,
\[
\lim_{|x|\to\infty}u(x)=0.
\]
This completes the proof.
\end{proof}

\subsection{Existence theory for viscosity solutions}\label{sec:existence of vs}
According to the above regularity theorem, we prove the weak solutions obtained in Theorem \ref{thm:1.1} are  viscosity solutions. First, we present the relevant stability results for viscosity solutions that are required in the proof process.

\begin{definition}[$\GC$-convergence]
   A sequence of lower-semicontinuous functions $u_{k}$ \emph{$\GC$-converges} to $u$ in  $\mathbb{R}^n$ if the two following conditions hold:

\begin{itemize}
    \item for every sequence $x_{k} \to x$ in $\mathbb{R}^n$, $\liminf_{k \to \infty} u_{k}(x_{k}) \geqslant u(x)$.
    \item for every $x \in \mathbb{R}^n$, there is a sequence $x_{k} \to x$ in $\mathbb{R}^n$ such that 
    $$
    \limsup_{k \to \infty} u_{k}(x_{k}) = u(x).
    $$
\end{itemize}
\end{definition}
It is natural that a uniformly convergent sequence $u_{k}$ also converges in the $\Gamma$-convergence sense. A fundamental property of $\Gamma$-limits is that if $u_{k}$ $\Gamma$-converges to $u$ and $u$ has a strict local minimum at $x$, then there exists a sequence of points $x_{k} \to x$ such that $u_{k}$ has a local minimum at $x_{k}$.

\begin{lemma}\label{stab}
Let  $u_{k}$ be a sequence of functions that are uniformly bounded in $\R^{n}$ and lower-semicontinuous in $\R^n$ such that

\begin{enumerate}
    \item[(i)] $-\Delta u_k+(-\Delta)^s u_{k} \geqslant f_{k}$ in $\R^n$,
    \item[(ii)] $u_{k} \to u$ in the $\GC$ sense in $\R^n$,
    \item[(iii)] $u_{k} \to u$ a.e. in $\R^{n}$, 
    \item[(iv)] $f_{k} \to f$ locally uniformly in $\R^n$ for some continuous function $f$.
\end{enumerate}

Then $-\Delta u+(-\Delta)^s u \geqslant f$ in $\R^n$.
\end{lemma}

\begin{proof}
Let $\varphi$ be a test function from below for $u$ touching at a point $x$ in a neighborhood $N$. Since $u_{k}$ $\GC$-converges to $u$ in $\R^n$, for large $k$ we can find $x_{k}$ and $d_{k}$ such that $\varphi + d_{k}$ touches $u_{k}$ at $x_{k}$. Moreover, $x_{k} \to x$ and $d_{k} \to 0$ as $k \to +\infty$. 

Since $-\Delta u_k+(-\Delta)^s u_{k} \geqslant f_{k}$, if we let
\[
v_{k} = 
\begin{cases}
\varphi + d_{k} & \text{in } N, \\
u_{k} & \text{in } \R^{n} \setminus N,
\end{cases}
\]
we have $-\Delta v_k(x_k)+(-\Delta)^s v_{k}(x_k) \geqslant f_{k}(x_{k}, v_k(x_k))$.

Next, we prove that $-\Delta v_k(z)+(-\Delta)^s v_k(z)\to -\Delta v(z)+(-\Delta)^s v(z)$ is locally uniform in $N.$ 
Since $v_k-v=d_k$ in $N$ and (iii) and \cite[Lemma 4.5]{MR2494809}, we have $(-\Delta+(-\Delta)^s )v_k\to (-\Delta+(-\Delta)^s )v$ locally uniformly in $N.$ 

Finally, we have that $-\Delta v+(-\Delta)^s v$ is continuous in $N$. By computing, we have
\begin{align*}
&\left|(-\Delta) v_k(x_k)+(-\Delta)^s v_k(x_k)-(-\Delta) v(x)+(-\Delta)^s v(x)\right|\leqslant \left|\left(-\Delta+(-\Delta)^s\right)v_k(x_k)-\left((-\Delta+(-\Delta)^s)\right)v(x_k)\right|\\
&\quad+\left|\left(-\Delta+(-\Delta)^s\right)v(x_k)-\left((-\Delta+(-\Delta)^s)\right)v(x)\right|\to0.
\end{align*}
So $-\Delta v_{k}(x_{k})+(-\Delta)^s v_k(x_k)$ converges to $-\Delta v(x)+(-\Delta)^s v(x)$ as $k \to \infty$. Since $x_{k} \to x$ and $f_{k} \to f$ locally uniformly, we also have $f_{k}(x_{k},v_k(x_k)) \to f(x, v(x))$, which finally implies $-\Delta v(x)+(-\Delta)^s v(x) \geqslant f(x, v(x))$.
\end{proof}

\begin{proof}[\textbf{Proof of Theorem \ref{thm:vis}}]
By Lemma \ref{regu} and Theorem \ref{thm:1.1}, we have:
 $$u \in C(\R^n),$$ which satisfies the basic continuity requirement for the definition of viscosity solutions.


Let $f(x,u(x)):=\lambda h u^p+u^{2^*-1}.$ Now, let us fix $x_0\in\mathbb{R}^n.$ We consider the mollifier $\rho_{\epsilon}(x):=\epsilon^{-n}\rho(x/\epsilon),\epsilon>0.$ We define $u_{\epsilon}:=u*\rho_{\epsilon}$ and $f_{\epsilon}:=f*\rho_{\epsilon}.$ And we  recall the stability of viscosity solutions (see Lemma \ref{stab}) we may suppose that $u$ is also smooth in a neighborhood $\Omega$ of $x_0$ and, from \eqref{equu:2.5}, for any  $\varphi\in C_0^{\infty}(\Omega)$ 
\begin{equation*}
		\begin{aligned}&\int_{\R^n}-\Delta u \varphi dx+\int_{\R^{2n}}(-\Delta)^s u \varphi dx\\
			&=\int_{\mathbb{R}^n}\triangledown u(x)\cdot\triangledown \varphi(x)\, dx+\iint_{\mathbb{R}^{2n}}\frac{\left(u(x)-u(y)\right)\left(\varphi(x)-\varphi(y)\right)}{|x-y|^{N+2s}}\, dxdy\\
			&\,\,\,\,\,=\int_{\mathbb{R}^n}f(x,u(x)) \varphi \, dx.
		\end{aligned}
	\end{equation*}
Therefore,
$$-\Delta u+(-\Delta)^s u=f(x,u(x))$$
 for almost any $x\in\Omega$ and, in fact, for any $x\in\Omega.$ In particular, $-\Delta u(x_0)+(-\Delta)^s u(x_0)=f(x_0,u(x_0)).$
 
Then take a test function $\phi \in C^2(B_R(x_0))$ such that:
\[
\phi(x_0) = u(x_0), \quad \phi(x) \geqslant u(x) \quad \forall x \in B_R(x_0).
\]
Define a function $v$ by
\[
v(x) = 
\begin{cases}
\phi(x), & x \in B_R(x_0), \\
u(x), & x \in \R^n \setminus B_R(x_0).
\end{cases}
\]
We need to prove:
\[
-\Delta v(x_0) + (-\Delta)^s v(x_0) \leqslant f(x_0,v(x_0)).
\]

 At the minimum point, we have  $ -\Delta v(x_0) = -\Delta \phi(x_0),$ then 
  \begin{equation}\label{equ4.2.}
  -\Delta v(x_0) \leqslant -\Delta u(x_0). 
  \end{equation}

Since $v\geqslant u$ in $\mathbb{R}^n$ and $v(x_0)=u(x_0)$, we obtain
  \begin{equation}\label{equ4.3.}
  \begin{aligned}
  (-\Delta)^s v(x_0) &= c(n,s) \int_{\R^n} \frac{2v(x_0) - v(x_0+y) - v(x_0-y)}{|y|^{n+2s}} \, dy\\
  &\leqslant c(n,s) \int_{\R^n} \frac{2u(x_0) - u(x_0+y) - u(x_0-y)}{|y|^{n+2s}} \, dy.
  \end{aligned}
  \end{equation}
  Combining \eqref{equ4.2.} and \eqref{equ4.3.}, we get that $-\Delta v(x_0)+(-\Delta)^s v(x_0)\leqslant -\Delta u(x_0)+(-\Delta)^s u(x_0)=f(x_0, v(x_0)).$
   This shows that $u$ is a viscosity subsolution, and one can prove similarly that $u$ is a viscosity supersolution.

Since $u$ is both a viscosity subsolution and a viscosity supersolution of equation \eqref{equ1}, we conclude that $u$ is a viscosity solution of \eqref{equ1}.
\end{proof}

\section{Qualitative properties of positive classical solutions}
In this section, we will discuss the power-type decay estimates and radial symmetry of the positive classical solutions to equation \eqref{equ1}. In Section \ref{sec5}, we apply the maximum principle and comparison arguments to obtain the decay at infinity of positive classical solutions. In Section \ref{sec6}, we employ the method of moving planes to show that these classical solutions are also radially symmetric.

\subsection{Power-type decay of positive solutions}\label{sec5}
In this section, we shall apply the maximum principle and comparison arguments to obtain the decay at infinity of positive classical solutions.

First, we devote  to establishing
the existence of positive classical solutions.
\begin{theorem}\label{thm5.7}
Let \(n\geqslant 4\) and \(s\in(0,1)\).

Then, problem \eqref{equ1} has a classical solution, which satisfies \(u>0\) in \(\mathbb{R}^{n}\). Moreover,
\begin{equation}\label{5.10equ}
\lim_{|x|\to+\infty}u(x)=0. 
\end{equation}
\end{theorem}

\begin{proof}
From Theorem~\ref{thm:1.2} and \cite{Su-Xu}, we have that there exists a nonnegative classical solution of \eqref{equ1}.
Now, we focus on proving that \(u>0\) in \(\mathbb{R}^{n}\). For this, we argue by contradiction, and we assume that there exists a global minimum point \(x_{0}\in\mathbb{R}^{n}\) at which \(u(x_{0})=0\). Accordingly, we have that \(\Delta u(x_{0})\geqslant 0\) and \((-\Delta)^{s}u(x_{0})<0\). As a result, we deduce from \eqref{equ1}, $h>0$ and $\lambda>0$ that
\[
0=\lambda hu^{p}(x_{0})+u^{2^*-1}(x_0)=-\Delta u(x_{0})+(-\Delta)^{s}u(x_{0})<0,
\]
which is a contradiction.

On the other hand, \eqref{5.10equ} follows immediately from Lemma \ref{regu}.
\end{proof}

To prove the decay estimate of positive solutions, we start with the following lemmas, which construct suitable subsolutions and supersolutions.	
\begin{lemma}\label{uouo} Let \(n\geqslant4\) and \(s\in(0,1)\). Define
\[
\omega:=\mathcal{Z}*\chi_{B_{1/2}}.
\]
Then $\omega\in \mathcal{X}^{1,2}(\mathbb R^n)\cap C^{1,\alpha}(\mathbb R^n) \text{ for every }\alpha\in(0,1).
$
Moreover, \(\omega>0\) in \(\mathbb R^n\), and there exist constants \(c,C>0\), depending only on \(n,s\), such that
\[
\frac{c}{(1+|x|)^{n-2s}}
\leqslant\omega(x)\leqslant
\frac{C}{(1+|x|)^{n-2s}}
\qquad\text{for all }x\in\mathbb R^n.
\]
\end{lemma}
\begin{proof} Since
\[
\chi_{B_{1/2}}
\in L^{\frac{2n}{n+2}}(\mathbb R^n)\cap L^\infty(\mathbb R^n),
\]
Lemma \ref{proposition} yields
\[
\omega\in \mathcal{X}^{1,2}(\mathbb R^n)\cap L^\infty(\mathbb R^n),
\quad
-\Delta\omega+(-\Delta)^s\omega=\chi_{B_{1/2}}\text{ in }\mathbb R^n
\]
in the weak sense. Applying the convolution estimates established in the proof of Proposition \ref{pro 4.1} imply
\[
\omega\in C^{1,\alpha}(\mathbb R^n)
\qquad\text{for every }\alpha\in(0,1).
\]
Furthermore, the strict positivity of \(\mathcal{Z}\) gives \(\omega>0\) in \(\mathbb R^n\).

For \(|x|\geqslant2\) and \(y\in B_{1/2}\),
\[
\frac34|x|\leqslant |x-y|\leqslant\frac54|x|.
\]
Thus, by Theorem \ref{kernelz},
\[
\begin{aligned}
c|x|^{-(n-2s)}
&\leqslant
\int_{B_{1/2}}\mathcal{Z}(x-y)\,dy=\omega(x)
\leqslant C|x|^{-(n-2s)}.
\end{aligned}
\]
By the continuity and strict positivity of \(\omega\) on \(\overline{B_2}\), we have
$c(1+|x|)^{-(n-2s)}
\leqslant\omega(x)\leqslant
C(1+|x|)^{-(n-2s)}
\qquad\text{for all }x\in\mathbb R^n.$
\end{proof}

\begin{proof}[\textbf{Proof of  the power-type decay in Theorem \ref{thm4.8}}]
We consider the function \( \omega \) given by Lemma \ref{uouo}. Utilizing the positivity and continuity of \( u \) and \( \omega \) in \( \mathbb{R}^n \), we can find a constant \( \beta > 0 \) such that \( h := u - \beta \omega > 0 \) in \( B_1 \).

Furthermore, from \eqref{equ1} and Lemma \ref{uouo}, we see that
\[
\begin{cases}
-\Delta h + (-\Delta)^s h \geqslant 0 & \text{in } \mathbb{R}^n \setminus B_1, \\
\lim_{|x| \to +\infty} h(x) = 0.
\end{cases}
\]

We claim that  
\begin{equation}\label{equ4.11}
h(x) \geqslant 0 \quad \text{for all } |x| > 1. 
\end{equation}

Indeed, let $h^- = \max(-h, 0)$ be the negative part of $u$. Since $h \in \mathcal{X}^{1,2}(\mathbb{R}^n)$, we have $h^- \in \mathcal{X}^{1,2}(\mathbb{R}^n)$. Using $h^-$ as a test function,
		\[
\begin{aligned}
0
&\leqslant\langle \mathcal{L}h,h^-\rangle\\
&=\int_{\mathbb R^n}\nabla h\cdot\nabla h^-\,dx+\frac{c_{n,s}}2
\iint_{\mathbb R^{2n}}
\frac{(h(x)-h(y))(h^-(x)-h^-(y))}
{|x-y|^{n+2s}}\,dx\,dy\\
&\leqslant-\|\nabla h^-\|_{L^2(\mathbb R^n)}^2
-\frac{c_{n,s}}2[h^-]_s^2.
\end{aligned}
\]
Hence $\langle \mathcal{L} h, h^- \rangle \leqslant 0$. But by assumption $\langle \mathcal{L} h, h^- \rangle \geqslant 0$, so we must have:
		\[
		\langle \mathcal{L} h, h^- \rangle = 0
		\]
		This forces $h^- \equiv 0$, thus $h \geqslant 0$ in $|x|\geqslant 1$.
Hence, for every \( |x| > 1 \),  
\[
u(x) \geqslant \beta \omega(x) \geqslant \frac{c_1}{|x|^{n-2s}}.
\]  
This establishes the bound from below in Theorem \ref{thm4.8}.  

%

We now focus on the bound from above. Set $V:=\lambda hu^{p-1}+u^{2^*-2}.$
Since \(h\in L^1(\mathbb R^n)\cap L^\infty(\mathbb R^n)\) and
\(u\in \mathcal{X}^{1,2}(\mathbb R^n)\cap L^\infty(\mathbb R^n)\),
\[
V\in L^{\frac{n}{2}}(\mathbb R^n)\cap L^\infty(\mathbb R^n).
\]
Choose
\[
\frac n2<r<\frac n{2s},
\quad r'=\frac r{r-1}.
\]
Then \(V\in L^r(\mathbb R^n)\), and the kernel estimates and Lemma \ref{uouo} imply $\mathcal{Z},\omega\in L^{r'}(\mathbb R^n).$
For \(R>1\), define
\[
T_R\phi(x):=
\int_{\mathbb R^n\setminus B_R}
\mathcal{Z}(x-y)V(y)\phi(y)\,dy.
\]
We claim that
\[
T_R\omega(x)
\leqslant C\|V\|_{L^r(\mathbb R^n\setminus B_R)}\omega(x),
\]
where \(C\) is independent of \(R\) and \(x\). Indeed, if
\[
|x-y|<\frac{1+|x|}{2},
\]
then \(\omega(y)\leqslant C\omega(x)\). Otherwise, the kernel estimates give
\[
\mathcal{Z}(x-y)\leqslant C(1+|x|)^{-(n-2s)}\leqslant C\omega(x).
\]
Consequently,
\[
\mathcal{Z}(x-y)\omega(y)
\leqslant C\omega(x)\bigl(\mathcal{Z}(x-y)+\omega(y)\bigr),
\]
and the H\"older's inequality proves the claim.

For $\phi\in \mathcal{X}^{1,2}(\mathbb R^n),$ we have $\|\phi\|_{L^{2^*}}^2\leqslant C_s\|\nabla\phi\|_2^2.$
Since \(V\in L^r\cap L^{n/2}\), we fix \(R>1\) sufficiently large that
\[
T_R\omega\leqslant\frac12\omega,
\quad
C_s\|V\|_{L^{n/2}(\mathbb R^n\setminus B_R)}<1.
\]
Define
\[
v_0:=\omega,\quad
v_{k+1}:=\omega+T_Rv_k.
\]
We have
\[
v_1\geqslant v_0,\quad
v_{k+1}-v_k=T_R(v_k-v_{k-1})\geqslant0.
\]
Moreover, if \(v_k\leqslant2\omega\), we have
\[
v_{k+1}\leqslant\omega+2T_R\omega\leqslant2\omega.
\]
Therefore
\[
\omega\leqslant v_k\leqslant v_{k+1}\leqslant2\omega.
\]
Letting \(v:=\lim_{k\to\infty}v_k\),  we have
\[
\omega\leqslant v\leqslant2\omega,
\quad
v=\omega+T_Rv.
\]
Since \(\omega\in L^{2^*}\cap L^\infty\), we have
\[
v\in L^{2^*}\cap L^\infty.
\]
Furthermore,
\[
\|Vv\,\chi_{\mathbb R^n\setminus B_R}\|_{L^{\frac{2n}{n+2}}}
\leqslant
\|V\|_{L^{n/2}(\mathbb R^n\setminus B_R)}
\|v\|_{L^{2^*}},
\]
and \(Vv\,\chi_{\mathbb R^n\setminus B_R}\in L^\infty\). Therefore, by $v=\mathcal{Z}*\left(\chi_{B_{1/2}}+Vv\,\chi_{\mathbb R^n\setminus B_R}\right)$
and Lemma \ref{proposition}, we obtain
\[
v\in \mathcal{X}^{1,2}(\mathbb R^n),\quad
\mathcal{L}v=\chi_{B_{1/2}}
+Vv\,\chi_{\mathbb R^n\setminus B_R}.
\]

Since \(v\geqslant\omega>0\), choose \(\gamma>0\) sufficiently small that
\[
g:=v-\gamma u\geqslant0\qquad\text{in }B_R.
\]
Using \(g^-\) as a test function, we obtain
\[
\langle \mathcal{L}g,g^-\rangle
=\int_{\mathbb R^n\setminus B_R}Vg\,g^-\,dx
=-\int_{\mathbb R^n\setminus B_R}V(g^-)^2\,dx.
\]
On the other hand, we have
\[
\langle \mathcal{L}g,g^-\rangle
\leqslant-\|\nabla g^-\|_2^2-\frac{c_{n,s}}2[g^-]_s^2.
\]
Hence
\[
\begin{aligned}
\|\nabla g^-\|_2^2+\frac{c_{n,s}}2[g^-]_s^2\leqslant\int_{\mathbb R^n\setminus B_R}V(g^-)^2\,dx\leqslant C_S\|V\|_{L^{n/2}(\mathbb R^n\setminus B_R)}
\|\nabla g^-\|_2^2.
\end{aligned}
\]
By the choice of \(R\), this implies \(g^-\equiv0\). Consequently,
\[
u\leqslant\gamma^{-1}v\leqslant2\gamma^{-1}\omega.
\]
Applying Lemma \ref{uouo}, we conclude that $u(x)\leqslant\frac{C_2}{|x|^{n-2s}}\text{ for every }|x|\geqslant1.$

\end{proof}

\subsection{Radial symmetry of positive solutions}\label{sec6}
Our aim is now to prove the radial symmetry of the positive solution (Theorem \ref{thm4.8}). The main statement of this section is the following.

\begin{theorem}\label{thm5.1}
Let \( n \geqslant 4\) and \( (h_1)\, (h_2)\) hold. Then, all positive solutions of \eqref{equ1} are radially symmetric about some point in $\mathbb{R}^n$ when $p\in[1, 2^*-1)$.
\end{theorem}

To establish Theorem \ref{thm5.1}, we will exploit the moving planes method. To begin with, we recall some notations. For any \(\lambda \in \mathbb{R}\), we set
\[
\Sigma_{\lambda} := \{x \in \mathbb{R}^n \text{ s.t. } x_1 > \lambda\}
\]
and
\[
T_{\lambda} := \{x \in \mathbb{R}^n \text{ s.t. } x_1 = \lambda\}.
\]
Moreover, for any \( x = (x_1, x_2, \dots, x_n) \in \mathbb{R}^n \), we set \( x^\lambda := (2\lambda - x_1, x_2, \dots, x_n) \), \( u_\lambda(x) := u(x^\lambda) \) and $w_{\lambda}(x)=u_{\lambda}(x)-u(x)$.

\begin{proof}[\textbf{Proof of the radial symmetry in Theorem \ref{thm4.8}}]
Similar to the proof of \cite[Theorem 4.6]{SerenaDipierro2025} and \cite[Theorem 3.1]{MR3626036}, we can obtain the radially symmetry of classical solutions.We omit the details. \(\square\)
\end{proof}

\bibliography{reference}{}
\bibliographystyle{plain}

\end{document}